\documentclass[10pt,a4paper]{amsart}
\usepackage[T1]{fontenc}
\usepackage[utf8]{inputenc}
\usepackage[italian, english]{babel}
\usepackage{geometry}
\usepackage{enumitem}
\usepackage{tikz-cd}
\usepackage[hidelinks]{hyperref}
\usepackage[timezonesep={\ UTC}, showseconds=false]{datetime2}

\usepackage{amsmath}
\usepackage{amssymb}
\usepackage{amsfonts}
\usepackage{mathtools}
\usepackage{amsthm}
\usepackage{braket}
\usepackage{mathrsfs}
\usepackage{dsfont}
\usepackage{bbm}
\usepackage{stmaryrd}
\usepackage{tensor}
\usepackage{extarrows}

\renewcommand{\setminus}{\smallsetminus}
\newcommand{\allora}{\Longrightarrow}
\DeclareMathOperator{\immagine}{im}
\newcommand{\im}{\immagine}
\DeclareMathOperator{\Impart}{Im}
\renewcommand{\Im}{\Impart}
\newcommand{\inj}{\hookrightarrow}
\newcommand{\surj}{\twoheadrightarrow}

\newcommand{\iso}{\xrightarrow{\sim}}

\newcommand{\longiso}{\overset{\sim}{\longrightarrow}}

\newcommand{\numberset}{\mathbb}

\newcommand{\Z}{\numberset{Z}}
\newcommand{\Q}{\numberset{Q}}
\newcommand{\R}{\numberset{R}}
\newcommand{\C}{\numberset{C}}

\renewcommand{\phi}{\varphi}
\renewcommand{\theta}{\vartheta}
\renewcommand{\epsilon}{\varepsilon}

\DeclareFontFamily{U}{wncy}{}
\DeclareFontShape{U}{wncy}{m}{n}{<->wncyr10}{}
\DeclareSymbolFont{mcy}{U}{wncy}{m}{n}
\DeclareMathSymbol{\sha}{\mathord}{mcy}{"58}

\newcommand{\F}{\mathbb{F}}
\newcommand{\id}{\mathrm{id}}

\DeclareMathOperator{\coker}{coker}
\newcommand{\trasp}[1]{\tensor*[^{t}]{#1}{}}
\newcommand{\traspinv}[1]{\tensor*[^{t}]{#1}{^{-1}}}

\DeclareMathOperator{\End}{End}
\DeclareMathOperator{\Hom}{Hom}
\DeclareMathOperator{\GL}{GL}
\DeclareMathOperator{\SL}{SL}
\DeclareMathOperator{\Gal}{Gal}
\newcommand{\rotdxsim}{\rotatebox{-90}{$\sim$}}

\newcommand{\rotsupseteq}{\rotatebox{-90}{$\supseteq$}}
\renewcommand{\projlim}{\varprojlim}
\renewcommand{\injlim}{\varinjlim}

\newcommand{\cont}{\textup{cont}}
\newcommand{\crys}{\textup{crys}}
\newcommand{\ur}{\textup{ur}}
\newcommand{\ord}{\textup{ord}}
\newcommand{\str}{\textup{str}}

\newcommand{\iw}{\textup{Iw}}
\newcommand{\et}{\textup{\'{e}t}} 
 
\newcommand{\fs}{\textup{fs}}
\newcommand{\divisible}{\textup{div}}
\DeclareMathOperator{\cohomology}{H}
\DeclareMathOperator{\tildecohomology}{\widetilde{H}}
\newcommand{\etcohomology}{\cohomology_\et}
\newcommand{\hone}{\cohomology^1}
\newcommand{\honef}{\cohomology^1_f}
\newcommand{\honetildef}{\tildecohomology^1_f}
\newcommand{\honetildefiw}{\tildecohomology^1_{f, \iw}}
\newcommand{\honeur}{\cohomology^1_\ur}
\newcommand{\hones}{\cohomology^1_s}
\newcommand{\ho}{\cohomology^0}
\DeclareMathOperator{\res}{\mathrm{res}}
\DeclareMathOperator{\cores}{\mathrm{cores}}
\newcommand{\contcochains}{\mathrm{C}_\cont}
\newcommand{\contcohocomplex}{\contcochains^\bullet}
\newcommand{\contcoho}{\contcohocomplex}

\DeclareMathOperator{\chow}{CH}
\newcommand{\AJ}{\mathrm{AJ}}
\newcommand{\imajt}{\widetilde{\Lambda}_\p}
\newcommand{\shat}{\widetilde{\sha}_{\p^\infty}}
\newcommand{\cl}{\mathrm{cl}}
\DeclareMathOperator{\Ext}{Ext}
\DeclareMathOperator{\red}{red}
\DeclareMathOperator{\Tr}{Tr}

\newcommand{\Gr}{\textup{Gr}}

\DeclareMathOperator{\infl}{infl}

\renewcommand{\O}{\mathcal{O}}
\newcommand{\K}{\mathcal{K}}
\newcommand{\p}{\mathfrak{p}}
\newcommand{\mfrak}{\mathfrak{m}}
\newcommand{\afrak}{\mathfrak{a}}
\newcommand{\Nfrak}{\mathfrak{N}}
\DeclareMathOperator{\Frob}{Frob}
\newcommand{\calF}{\mathcal{F}}
\newcommand{\calX}{\mathcal{X}}
\newcommand{\calY}{\mathcal{Y}}
\newcommand{\calZ}{\mathcal{Z}}
\newcommand{\calI}{\mathcal{I}}
\newcommand{\triv}{\mathds{1}}
\newcommand{\calG}{\mathcal{G}}
\DeclareMathOperator{\charpol}{char}

\newcommand{\ccat}{\mathscr{C}}

\DeclareMathOperator{\D}{D}
\newcommand{\diff}{\mathrm{d}}
\DeclareMathOperator{\Cone}{Cone}
\newcommand{\ringmod}[1]{(#1\text{-} \mathrm{Mod})}
\newcommand{\rmod}{\ringmod{R}}
\newcommand{\rgmod}{\ringmod{R[G]}}
\newcommand{\rgksmod}{\ringmod{R[G_{K, S}]}}
\newcommand{\omod}{\ringmod{\O}}
\newcommand{\ogksmod}{\ringmod{\O[G_{K,S}]}}
\newcommand{\lambdamod}{\ringmod{\Lambda}}
\newcommand{\ad}{\textup{ad}}
\newcommand{\indad}{\textup{ind-ad}}
\newcommand{\ft}{\textup{ft}}
\newcommand{\coft}{\textup{coft}}

\newcommand{\ev}{\mathrm{ev}}
\newcommand{\adj}{\mathrm{adj}}

\DeclareMathOperator{\selmers}{\widetilde{\mathbb{R}\Gamma}}
\newcommand{\selmer}{\selmers_f}
\newcommand{\selmeriw}{\selmers_{f, \iw}}
\newcommand{\kugasato}[1]{\tilde{\mathcal{E}}_{#1}}
\newcommand{\kugasatoplain}[2]{\kugasato{\Gamma(#1)}^{#2}}

\DeclareMathOperator{\corr}{Corr}
\newcommand{\Isog}{\mathrm{Isog}}
\DeclareMathOperator{\Graphfun}{Graph}
\newcommand{\squarefree}{\mathscr{K}}
\DeclareMathOperator{\loc}{loc}
\newcommand{\dkol}[1]{\mathcal{D}(#1)}
\theoremstyle{definition}
\newtheorem{definition}{Definition}[section]
\theoremstyle{plain}
\newtheorem{theorem}[definition]{Theorem}
\newtheorem{lemma}[definition]{Lemma}
\newtheorem{proposition}[definition]{Proposition}
\newtheorem{corollary}[definition]{Corollary}		
\theoremstyle{remark}
\newtheorem{remark}[definition]{\textsc{Remark}}
\newtheorem{example}[definition]{Example}
\newtheorem{assumption}[definition]{\textsc{Assumption}}

\usepackage[autostyle]{csquotes}
\usepackage[backend=biber, style=alphabetic-verb, minnames=5, maxnames=9]{biblatex}
\DefineBibliographyExtras{english}{}

\DeclareFieldFormat{postnote}{#1}
\DeclareFieldFormat{multipostnote}{#1}

\title[$\shat(f/K) = 0$ and Anticyclotomic Iwasawa Theory]{Vanishing of the $\p$-part of the Shafarevich--Tate group of a modular form and its consequences for Anticyclotomic Iwasawa Theory}
\author{Luca Mastella}
\date{}
\thanks{}
\address{}
\email{luca.mastella@gmail.com}

\begin{document}

\maketitle

\begin{abstract}
In this article we prove a refinement of a theorem of Longo and Vigni in the anticyclotomic Iwasawa theory for modular forms. 
More precisely we give a definition for the ($\p$-part of the) Shafarevich--Tate groups $\shat(f/K)$ and $\shat(f/K_\infty)$ 
of a modular form $f$ of weight $k >2$,  over an imaginary quadratic field $K$ satisfying the \emph{Heegner hypothesis} 
and over its \emph{anticyclotomic $\Z_p$-extension} $K_\infty$ and we 
show that if the \emph{basic generalized Heegner cycle} $z_{f, K}$ is non-torsion and not divisible by $p$,  then 
$\shat(f/K) = \shat(f/K_\infty) = 0$.
\end{abstract}

Let $f = \sum_{n > 0} a_n q^n$ be a cuspidal newform of even weight $k > 2$ and level $\Gamma_0(N)$, fix an odd prime $p \nmid N$ 
and an embedding $i_p \colon \bar{\Q} \inj \bar{\Q}_p$. Denote by $F$ the totally real field generated over $\Q$ by the Fourier coefficients $a_n$ of $f$ 
and let $\O_F$ be its ring of integers. The embedding $F \inj \bar{\Q}_p$ induced by $i_p$ defines a prime ideal $\p$ of $\O_F$ above $p$: 
let $\K := F_\p$ be the completion of $F$ at $\p$ and let $\O$ be its valuation ring. 
Deligne (see \cite[]{deligne:formes-modulaire}) attached to $f$ and $\p$ a $p$-adic Galois representation $W_\p$ of $G_\Q = \Gal(\bar{\Q}/\Q)$, 
that is a $2$-dimensional $\K$-vector space endowed with a continuous $G_{\Q}$-action. Let $V = W_\p(k/2)$ be the $k/2$-th Tate twist of $W_\p$ and let $T$ 
be the self-dual lattice inside $V$, defined in \cite{nekovar:kuga-sato}; set $A = V/T$. 

Fix moreover an imaginary quadratic field $K$ of discriminant $d_K \ne -3, -4$ coprime with $Np$ and satisfying the so called \emph{Heegner hypothesis}, 
i.e.~such that the prime factors of $N$ split in $K$, and let $K_\infty$ be its anticyclotomic $\Z_p$-extension, 
that is the unique $\Z_p$-extension of $K$ which is pro-dihedral over $\Q$. Put $\Gamma = \Gal(K_\infty/K) \cong \Z_p$ and let 
$\Lambda = \O \llbracket \Gamma \rrbracket \cong \O \llbracket T \rrbracket$ be the Iwasawa algebra. 

The main result of \cite{longo-vigni:generalized} is a structure theorem (as $\Lambda$-module) for the Pontryagin dual $\calX_\infty$ 
of the Bloch--Kato Selmer group $\honef(K_\infty, A)$ of $A$ over $K_\infty$: they show, under some hypothesis on $(f, K, \p)$ as in particular the 
$\p$-ordinarity of $f$ and the big image property for $V_\p := W_\p^\ast$, that $\calX_\infty$ is pseudo-isomorphic to $\Lambda \oplus M \oplus M$, 
for a torsion $\Lambda$-module $M$, moreover they formulate an anticyclotomic main conjecture in this setting  
and they prove one divisibility of it. 

In particular this shows that $\honef(K_\infty, A)$ has corank $1$ as $\Lambda$-module; the aim of this paper is to show that 
(under some technical assumptions) if the \emph{basic generalized Heegner cycle} $z_{f, K}$ (see Section \ref{sec:shafarevich-tate-mod-forms}) 
is not divisible by $p$ in $\hone(K, T)$, then $\honef(K_\infty, A)$ is in fact cofree of corank $1$ over $\Lambda$. 
We give moreover a suitable definition of the ($\p$-part of the) Shafarevich--Tate groups $\shat(f/K), \shat(f/K_\infty)$ of $f$ over $K$ and $K_\infty$ 
(see Definition \ref{def:image-abel-jacobi-K}, \ref{def:shat-K-infty}) in terms of the $p$-adic Abel-Jacobi maps 
\[
\AJ^\et_{K[n]} \colon \chow^{k-1}(X_{k-2}/K[n])_0  \otimes \O \to \hone(K[n], T),
\]
where $K[n]$ denote the ring class field of conductor $n > 1$ and $X_{k-2}$ is the \emph{generalized Kuga--Sato variety} of \cite{bdp:generalized}: 
the product of the Kuga--Sato variety $W_{k-2}$ of level $\Gamma_1(N)$ and dimension $k-1$, i.e.~the canonical desingularization of the $(k-2)$-fold self-fiber product of the universal generalized elliptic curve with 
$\Gamma_1(N)$-level structure, and the $(k-2)$-fold self-product of 
a fixed CM elliptic curve $A$ defined over $K[1]$. With this language the main theorem of this paper becomes:
\begin{theorem}[Theorem \ref{th:main}]\label{th:main-intro}
Under the assumptions of Section \ref{sec:framework}, suppose moreover that the basic generalized Heegner cycle $z_{f, K}$ is non-torsion and that 
$z_{f, K}$ is not divisible by $p$ in $\hone(K, T)$. Then $\shat(f/K)=0$ and 
\[
\honef(K, A) = z_{f, K} \cdot \K/\O,
\] 
moreover $\shat(f/K_\infty) = 0$ and $\honef(K_\infty, A)$ is cofree of corank $1$ over $\Lambda$.
\end{theorem}

The proof of Theorem \ref{th:main-intro} goes along the same lines of \cite{nekovar:kolyvagin}, which proves an analogous result for elliptic curves. 
In particular we first obtain the following vanishing result for $\shat(f/K)$, adapting to our context \cite{besser:finiteness-sha}, that used 
the \emph{(classical) Heegner cycles} of \cite{nekovar:kuga-sato}, while we need the generalized ones of \cite{bdp:generalized}.
\begin{theorem}[Theorem \ref{th:besser}]\label{th:besser-intro}
Let $p$ be a non-exceptional prime (Definition \ref{def:admissible-primes}) and $z_{f, K}$ be non-torsion in $\hone(K, T)$. Then 
\[
p^{2\calI_p} \shat(f/K) = 0,
\]
where $\calI_p$ is the smallest non-negative integer such that $z_{f, K}$ is nonzero in $\hone(K, A[p^{\calI_p + 1}])$. In particular, if $\calI_p = 0$, 
then $\shat(f/K)=0$ and $\honef(K, A) = z_{f, K} \cdot \K/\O$.  
\end{theorem} 

Then in Section \ref{sec:consequences} with an Iwasawa theoretical argument that follows that of \cite{nekovar:kolyvagin}, we show that the vanishing of $\shat(f/K)$ implies the vanishing of $\shat(f/K_\infty)$, proving in this way Theorem \ref{th:main-intro}.

\begin{remark}
Let us remark that the Heegner hypothesis on $K$ could probably be relaxed. Since $(N, d_K)$ we may write $N = N^+ N^-$, where $N^+$ and $N^-$ are coprime, 
and all the prime factors of $N^+$ (resp.~$N^-$) are split (resp.~inert) in $K$. Assume that $N^-$ is a square-free product of an even number of primes: 
under this condition, that we call \emph{generalized Heegner hypothesis}, one may construct 
(see \cite{brooks:shimura-curves-special-values, magrone:generalized-heegner-cycles}) 
generalized Heegner cycles over Shimura curves which enjoy the same formal properties of the generalized Heegner cycles considered in this paper. 
In a recent paper \cite{pati:generalized-quaternionic} Pati shows that replacing the generalized Heegner cycles with the ones 
over Shimura curves one can prove the result of \cite{longo-vigni:generalized} under the relaxed Heegner hypothesis. 
It may therefore be possible to prove in the same way results similar to those in this paper in this more general context. Partially this is carried out in a forthcoming paper with Antonio Lei and Luochen Zhao (see \cite[Section 3]{lei-mastella-zhao:bloch-kato-sha-anticyclotomic}, where we generalize a special case of Theorem \ref{th:besser-intro}). 
\end{remark}
%%%%%%%%%%%%%%%%%%%%%%%%%%%%%%%%%%%%%%%%%%%%%%%%%%%%%%%%%%%%%%%%%%%%%%%
\subsection{Content of this article}\label{sec:content}

Section \ref{sec:selmer-groups-and-complexes}-\ref{sec:mod-forms-gen-heegner-cycles} deal with the background material: 
in particular in Section \ref{sec:selmer-groups-and-complexes} we introduce the general notion of Selmer groups and Selmer complexes for $p$-adic representations;
in Section \ref{sec:anticyclotomic-iwasawa-theory} we introduce the anticyclotomic $\Z_p$-extension $K_\infty$ and we recall some notions of the theory 
of Semer group and complexes over $K_\infty$; in Section \ref{sec:mod-forms-gen-heegner-cycles} 
we introduce the self-dual representation attached to a modular form $f$ of weight $k \ge 2$ and the generalized Heegner cycles of 
Bertolini, Darmon and Prasanna.

Section \ref{sec:vanishing-of-sha} and Section \ref{sec:consequences} contain the new results: 
in Section \ref{sec:vanishing-of-sha} we prove Theorem \ref{th:besser-intro}, generalizing \cite{besser:finiteness-sha} using generalized Heegner cycles,  
in Section \ref{sec:consequences} we generalize the Iwasawa theoretic argument of \cite{nekovar:kolyvagin} 
and we combine it with Theorem \ref{th:besser-intro} in order to prove Theorem \ref{th:main-intro}.
%%%%%%%%%%%%%%%%%%%%%%%%%%%%%%%%%%%%%%%%%%%%%%%%%%%%%%%%%%%%%%%%%%%%%%%
\subsection{Acknowledgements}
The content of this article is part of my PhD Thesis at {Università degli Studi di Padova}. I would like to express my deepest gratitude to my advisor, Matteo Longo, for the nice topic he proposed to me and for his constant help and encouragement. Special thanks go to Stefano Vigni for several discussions and advice about the topics of this paper. I am grateful also to Francesco Zerman and Louchen Zhao for having pointed out some flaws in an earlier version of this article and to the anonymous referee for carefully reading the paper and suggesting many corrections and improvements. This article has been written during my stay at {Università degli Studi di Genova}, partially supported by PRIN 2022, The arithmetic of motives and L-functions. 
%%%%%%%%%%%%%%%%%%%%%%%%%%%%%%%%%%%%%%%%%%%%%%%%%%%%%%%%%%%%%%%%%%%%%%%
\subsection{Notations and Conventions}\label{sec:notations}

Throughout this paper we fix the following data:
\begin{itemize}
    \item a rational odd prime $p$;
    \item an integer $N > 5$ such that $(N, p)=1$;
    \item an imaginary quadratic field $K$ of discriminant $d_K \ne -3, -4$ (i.e.~such that $\O_K^\times = \set{\pm 1}$), such that $(d_K, pN) = 1$ and 
    satisfying the \emph{Heegner hypothesis}, i.e.~all prime factors of $N$ split in $K$;
    \item an algebraic closure $\bar{\Q}$ (resp.~$\bar{\Q}_\ell$ for any rational prime $\ell$) of $\Q$ (resp.~$\Q_\ell$);
    \item an embedding $i_\ell \colon \bar{\Q} \hookrightarrow \bar{\Q}_\ell$ for any rational prime $\ell$.
\end{itemize}
If $E$ is a number field let $G_E = \Gal(\bar{\Q}/E)$ be its absolute Galois group, 
moreover for any place $v \mid \ell$ of $E$ consider the completion $E_v$ of $E$ at $v$ and its absolute Galois group $G_v = \Gal(\bar{\Q}_\ell/E_v)$: 
note that the embedding $i_\ell$ realizes $G_v$ as a decomposition subgroup of $G_E$. 
The inertia subgroup of $G_E$ at $v$ will be denoted by $I_v$ and identified with $\Gal(\bar{\Q}_\ell/E_v^\ur) \subseteq G_v$. 
When $v$ is archimedean, i.e.~$E_v = \C$ or $\R$, define $I_v = G_v$, it is trivial or has order $2$. 
For $\ell \ne \infty$ we will use the notations $G_\ell = \Gal(\bar{\Q}_\ell/\Q_\ell)$ and $I_\ell$ for its inertia subgroup; 
$\Frob_\ell$ will denote an (arithmetic) Frobenius, i.e.~(a lift of) a topological generator of $G_\ell/I_\ell$. 
%while $\Frob_\ell^\geo$ a geometric Frobenius. 
In particular we fix the Frobenius automorphism $\Frob_p \in G_p$ such that $\chi_p(\Frob_p) =1$, where $\chi_p$ is the $p$-adic cyclotomic character (see Example \ref{ex:cyclotomic}): this is possible since  $\Q_p(\mu_{p^n})/\Q_p$ is totally 
ramified for any $n > 0$ and hence we may choose $\Frob_p$ fixing $\mu_{p^\infty}$.

If $S$ is a finite set of places of $E$ containing all archimedean places and all primes above $p$ let $E_S$ be the maximal extension of $E$ 
unramified outside $S$ and denote $G_{E, S} = \Gal(E_S/E)$. Denote by $I_v^S$ and $G_v^S$ respectively the decomposition and the inertia group of $G_{E, S}$ at $v$.

For any local field $E_v$ with residue characteristic $\ell$, let $E_v^t$ denote the maximal tamely ramified extension of $E_v$, 
i.e.~the union of all the algebraic extensions of $E_v$ with prime-to-$\ell$ ramification index.  

If $E'/E$ is an extension of number fields and $M$ is a $G_E$-module let 
\[
\res_{E'/E} \colon \hone(E, M) \to \hone(E', M), \qquad \cores_{E'/E} \colon \hone(E', M) \to \hone(E, M)
\]
denote the restriction and corestriction maps. Recall that if $E'/E$ is finite and Galois, we have 
\[
\res_{E'/E} \circ \cores_{E'/E} = \Tr_{E'/E} := \hspace{-15pt}\sum_{\sigma \in \Gal(E'/E)} \hspace{-15pt}\sigma \, \,, \qquad 
\cores_{E'/E} \circ \res_{E'/E} = [E' : E],
\]
 where $\sigma \in \Gal(E'/E)$ acts on $\hone(E', M)$ via the standard Galois action on cohomology.

For an abelian category $\ccat$, we denote by $\D^\ast(\ccat)$ the  
(unbounded for $\ast = \emptyset$, bounded if $\ast = b$, bounded below if $\ast = +$, bounded above if $\ast = -$ ) derived category of $\ccat$.
Let $R$ be a commutative ring, $\rmod$ denotes the category of $R$-modules; $\rmod_\ft$ (resp.~$\rmod_\coft$) the full subcategory of modules of finite type,
i.e.~noetherian (resp.~of cofinite type, i.e.~artinian). 
We denote by $\D^\ast_{\ft}\rmod$ (resp.~$\D^\ast_{\coft}\rmod$) the full subcategory of $\D^\ast\rmod$ 
consisting of complexes with cohomology of finite (resp.~cofinite) type.
In Section \ref{sec:selmer-complexes} we will fix a noetherian complete local ring $R$ of dimension $d$, maximal ideal $\mfrak$ and finite residue field $k$ of characteristic $p$: for a profinite group $G$ we denote  by $\rgmod^\ad$ (resp.~$\rgmod^\indad$) the category of admissible (resp.~ind-admissible)
$R[G]$-modules as defined in \cite[Definition 3.2.1, 3.3.1]{nekovar:selmer-complexes}.
%%%%%%%%%%%%%%%%%%%%%%%%%%%%%%%%%%%%%%%%%%%%%%%%%%%%%%%%%%%%%%%%%%%%%%%
\newpage
\section{Selmer groups and complexes}\label{sec:selmer-groups-and-complexes}

In this section we introduce the notion of $p$-adic Galois representation and its Selmer groups.

\subsection{$p$-adic Galois representations}\label{sec:p-adic-reps}

In this section let $E$ be a number field and fix a finite extension $\K/\Q_p$, denote by $\O$ its ring of integers, by $\p$ its maximal ideal, 
by $\pi$ a fixed uniformizer and by $\kappa$ its residue field.
\begin{definition}\label{def:p-adic-rep}
An $\O$-adic (resp.~$\K$-adic) representation of $G_E$ is a free $\O$-module $T$ (resp.~a $\K$-vector space $V$) of finite rank (resp.~dimension) with a
$\O$-linear (resp.~$\K$-linear) action of $G_E$ continuous with respect to the $\p$-adic topology on $T$ (resp.~$V$).
\end{definition}

These two notions are linked (and both usually referred as $p$-adic Galois representations). Indeed, if $T$ is an $\O$-adic representation, 
then the $K$-vector space $V = T \otimes_\O \K$ endowed with the induced action 
$\sigma \cdot (x \otimes r) = (\sigma \cdot x) \otimes r$, for  $x \in T$, $r \in \K$,
is a $\K$-adic representation of dimension equal to the rank of $T$. Vice versa any $G_E$-stable lattice $T$ inside a $\K$-adic representation $V$ is naturally 
an $\O$-adic representation, endowed with the induced action, of rank equal to the dimension of $V$.

Moreover we may attach to $T$ also a discrete torsion $G_E$-module $A = V/T = T \otimes_\O \K/\O$ 
and for any $M \in \O \setminus \O^\times$, we may consider its $M$-torsion submodule $A[M] = M^{-1}T/T \cong T/MT$, 
that is $G_E$-stable, by the $\O$-linearity of the action on $T$. In particular $A[p^n] \cong T/p^nT$ and 
\[
T = \projlim_{n} T/p^n \, T \cong \projlim_{n} A[p^n], \qquad A = \injlim_{n} A[p^n] = \bigcup_{n} A[p^n].
\]

\begin{remark}\label{rk:reps-as-homs}
Sometimes we will see these representations equivalently as the attached (continuous) group homomorphisms
\[
\rho_M \colon G_E \to \GL(M) \cong \GL_n(B); \quad \rho(\sigma)(x) = \sigma \cdot x \quad \text{for} \; \sigma \in G_E, x \in M,
\]
where $n$ is the rank of $T$ and $B = \O, \K$ or $\kappa$ respectively, depending on $M = T, V$ or $A[\pi]$.
\end{remark}

\begin{example}\label{ex:cyclotomic}
Let $\chi_p \colon G_\Q \to \Z_p^\times$ be the $p$-adic cyclotomic character, i.e.~the map such that $\sigma(\zeta) = \zeta^{\chi_p(\sigma)}$ for any 
$\sigma \in G_\Q$ and $\zeta \in \mu_{p^\infty}(\bar{\Q})$. The action of $G_\Q$ on a free $\Z_p$-module of rank one by $\chi_p$ gives rise to a $\Z_p$-adic 
representation. We will use the cyclotomic character moreover in order to twist a $p$-adic representation: let $T$ be a $\O$-adic representation of  
$G_E$, for a number field $E$, and denote the action of $\sigma \in G_E$ on $x \in T$ by $\sigma \cdot x$, 
then for any $j \in \Z$ we define $T(j)$ to be the $\O$-module $T$ as a $p$-adic representation of  $G_E$ with the action 
$\sigma \circ x = \chi_p(\sigma)^j(\sigma \cdot x)$. We will call it the $j$-th Tate twist of $T$.
\end{example}

An important feature of many $p$-adic Galois representations, in particular of all these of interest for this article, 
is that they are unramfied almost everywhere.
\begin{definition}\label{def:unramified-reps}
Let $M = T, V, A$, we say that $M$ is unramified at a place $v$ of $E$ if $I_v$ acts trivially, or equivalently if $\rho_M(I_v) = \set{1}$.
\end{definition} 

Note that $T$ is unramified if and only if $V$ or $A$ are so. 
\begin{example}\label{ex:cyclotomic-ramification}
Note that the $p$-adic character $\chi_p$ is unramified at any prime $\ell \ne p$. Indeed if $\zeta \in \mu_{p^\infty}(\bar{\Q})$, then $\Q_\ell(\zeta)$ 
is unramified and hence any $\sigma \in I_\ell = \Gal(\bar{\Q}_\ell/\Q_\ell^\ur)$ fixes $\zeta$. 
Therefore $\chi_p(I_\ell) = 1$. This is not the case for $\ell = p$, since 
$\Q_p(\zeta)$ in this case is totally ramified and hence there is a $\sigma \in I_p = \Gal(\bar{\Q}_p/\Q_p^\ur)$ that does not fix $\zeta$.
\end{example}

We will be interested in representations unramified outside a finite set $S$ of places of $E$ containing the archimedean places and all the places above $p$. 
This amounts to consider $p$-adic representations of $G_{E, S}$, by the following proposition:

\begin{proposition}\label{prop:factorization-maximal-unramified-extension}
A representation $\rho$ of $G_E$ is unramified outside $S$ if and only if it factors through $G_{E, S}$.
\end{proposition}

\begin{proof}
If $\rho$ factors as $\rho^S$ through $G_{E, S}$, then $\rho(I_v) = \rho^S \circ \pi(I_v) = \rho^S(I_v^S) = \rho^S(\set{1}) = \set{1}$ 
for any $v \notin S$, denoting by $\pi \colon G_E \surj G_{E, S}$ the quotient map. For the converse observe that applying \cite[Chapter II, Proposition 9.4]{neu:ant} to the extensions $E_S/E$ and $\bar{E}/E$ it follows that $G_{E, S}$ is the quotient of $G_E$ by $H$, the minimal closed normal subgroup containing all $I_v$ such that $v \notin S$. Thus $\ker \rho = H$ by minimality and therefore $\rho$ factors through $G_E/H \cong G_{E, S}$.
\end{proof}

Let us turn to the notion of duality between Galois representations. 
\begin{definition}\label{def:linear-dual}
The linear dual of an $\O$-adic (resp.~$\K$-adic) representation $T$ (resp.~$V$) of $G_E$ is defined to be the free $\O$-module $T^\ast = \Hom_{\O}(T, \O)$ 
(resp.~the $\K$-vector space $V^\ast = \Hom_\K(V, \K)$) endowed with the $\p$-adic topology and the continuous $G_E$-action  
$\sigma \cdot f(x) = f(\sigma^{-1} x)$  for any $\sigma \in G_E$, $f \in T^\ast$, $x \in T$ (resp.~$f \in V^\ast$, $x \in V$). 
\end{definition}
 
Note that the functor $(\, -\,)^\ast$ is a dualizing functor on the category of finite free $\O$-module (resp.~of $\K$-vector spaces), i.e.~the canonical map 
$T \to T^{\ast\ast}$ (resp.~$V \to V^{\ast \ast}$) is an isomorphism.  

For a compact or discrete $\O$-module $M$ we may define moreover its Pontryagin dual.
\begin{definition}\label{def:pontryagin}
Let $M$ be a compact or a discrete $\O$-module, the Pontryagin dual of $M$ is defined to be the $\O$-module $M^\vee = \Hom_\O^\cont(M, \K/\O)$ endowed with 
the compact-open topology. If moreover $M$ has a continuous $G_E$-action, it induces a continuous $G_E$-action on $M^\vee$ defined by 
$\sigma \cdot f(x) = f(\sigma^{-1} x)$  for any $\sigma \in G_E$, $f \in M^\vee$ and $x \in M$.
\end{definition}

\begin{remark}\label{rk:pontryagin-Z-p}
In the literature of $p$-adic Galois representation the most common definition of Pontryagin dual is $\Hom_{\Z_p}^\cont(M, \Q_p/\Z_p)$, but this is harmless,
since it is (non-canonically) isomorphic to ours, the isomorphism depending on the choice of a basis of $\O$ as finite free $\Z_p$-module.
\end{remark}

It is known (e.g.~\cite[Chapter I, Theorem 1.1.8]{neu:cnf}) that the $(\, -\,)^\vee$ functor defines a duality functor between the categories of compact and discrete $\O$-modules, i.e.~$M \to M^{\vee \vee}$ is an isomorphism for such modules. 

Starting from an $\O$-adic representation $T$, let $V = T \otimes \K$, $A = V/T$ as above and define moreover $A^\ast = V^\ast/T^\ast$. We may depict the relations among the linear and the Pontryagin dual representations of $T$ by the following diagram:
\[
\begin{tikzcd}[column sep=huge, row sep = huge]
T \ar[r, leftrightarrow, "(-)^\ast" {xshift=2pt}] \ar[d, "\otimes_\O \K/\O" '] \ar[rd, leftrightarrow, end anchor = {[xshift = 3pt, yshift=-3pt]}, "(-)^\vee" {xshift=-10pt, yshift=5pt}] & T^\ast \ar[d, xshift = -2pt, "\otimes_\O \K/\O"]\\
A \ar[ru, leftrightarrow] & A^\ast
\end{tikzcd}
\]

This diagram follows from the following proposition applied to $M = T, T^\ast$ since all morphisms $\phi \colon M \to \K/\O$ are continuous, 
as long as $M$ is a finite free $\O$-module. In fact, since $\phi$ is $\O$-linear, it is enough to check that $\phi$ is continuous at $0 \in M$; 
that is, as $M$ has the $\p$-adic topology and $\K/\O$ is discrete, we have to find an $N \ge 0$ such that $\phi(\p^N M) = 0 + \O$. 
Let therefore $m_1, \dots, m_s$ be a basis of $M$ over $\O$ and $\phi(m_i) = a_i/\pi^{n_i} + \O$ for $a_i \in \O^\times$, $n_i \ge 0$, let $N = \max_i \set{n_i}$,  
then $\phi(\p^N M) = 0 + \O$. 

\begin{proposition}\label{prop:iso-dualities}
If $M$ is a finite free $\O$-module, then $\Hom_\O(M, \O) \otimes_\O \K/\O \cong \Hom_\O(M, \K/\O).$ 
\end{proposition}

\begin{proof}
The isomorphism is given by the multiplication map $\phi \otimes x  \mapsto x \cdot \phi$. 
Indeed, let $m_1, \dots, m_s$ be an $\O$-basis of $M$ and write $m_i^\ast$ for the $\O$-linear maps  in $\Hom_\O(M, \O)$ such that $m_i^\ast(m_j) = \delta_{ij}$, 
where $\delta_{ij}$ is the Kronecker symbol. It is straight forward to check that the inverse of the previous map is given by the linear homomorphism 
$\psi \mapsto \sum_i m_i^\ast \otimes \psi(m_i)$, as any $\psi \in \Hom_\O(M, \K/\O)$ is uniquely defined by its image on a basis.
\end{proof}
%%%%%%%%%%%%%%%%%%%%%%%%%%%%%%%%%%%%%%%%%%%%%%%%%%%%%%%%%%%%%%%%%%%%%%%
\subsection{Selmer groups}\label{sec:selmer-groups}

In this paragraph $T$ will be an $\O$-adic representation of $G_{E, S}$, where $S$ is a finite set $S$ of places of $E$ containing the archimedean places and all the places above $p$, and let $V$ and $A$ be the induced representations as in the previous paragraph.
We introduce the general formalism of Selmer groups, following \cite{rubin:kolyvagin}. In the following all cohomology groups are the 
\emph{continuous cohomology}
groups, introduced in \cite{tate:k-2-gal-coho}. For the definition and the main properties of these groups  
we refer to \cite[B.2]{rubin:euler-systems} or \cite[Section II.3]{neu:cnf}.

\begin{definition}\label{sec:local-conditions}
A local condition $\calF= (\calF_v)_{v \in S}$ is the choice of a subspace $\hone_{\calF_v}(E_v, V) \subseteq \hone(E_v, V)$ for each $v \in S$. 
Note that the choice induces for any place $v \in S$ the submodules $\hone_{\calF_v}(E_v, A)$ of $\hone(E_v, A)$ 
and $\hone_{\calF_v}(E_v, T)$ of $\hone(E_v, T)$ 
taking respectively the image and the inverse image of $\hone_{\calF_v}(E_v, V)$ under the natural maps. 
The submodule $\hone_{\calF_v}(E_v, A[\pi^n])$ of $\hone(E_v, A[\pi^n])$ can be equally defined (see \cite[Remark I.3.9]{rubin:euler-systems}) as the image of 
$\hone_{\calF_v}(E_v, T)$ under the map induced by 
\[
\begin{tikzcd}
    T  \surj T/\pi^nT \ar[r, "\, \, \cdot \, \pi^{-n}", "\sim"'] & \pi^{-n}T/T = A[\pi^n]
\end{tikzcd}
\]
or as the inverse image of $\hone_{\calF_v}(E_v, A)$ by the map induced by the inclusion $A[\pi^n] \inj A$.
\end{definition}

\begin{example}\label{ex:unramified-finite}
Let $X = T, V, A, A[\pi^n]$ and $v$ a finite place of $E$. We define the subgroup of unramified cohomology classes as
\[
\hone_\ur(E_v, X) = \ker(\hone(E_v, X) \to \hone(I_v, X)),
\] 
where $I_v$ denotes the inertia subgroup of $G_v$. If $v \nmid p$ we define the \emph{finite} local condition at $v$ to be $\honef(E_v, V) = \hone_\ur(E_v, V)$ and, 
as in the previous definition, $\honef(E_v, A)$ and $\honef(E_v, T)$ to be respectively the image and the inverse image under the natural morphisms. 
In particular 
$\hone_\ur(E_v, T) \subseteq \honef(E_v, T)$ and $\honef(E_v, A) \subseteq \hone_\ur(E_v, A)$. The index
\[
 c_v(A) := [\hone_\ur(E_v, A) : \honef(E_v, A)]
\]
is finite and it is called the \emph{$p$-part of the Tamagawa number of $A$ at $v$}.  
If $V$ is unramified at $v$, then $c_v(A) = 1$ (see \cite[Lemma 3.5.\emph{iv}]{rubin:euler-systems}), i.e.~$\honef(E_v, A) = \honeur(E_v, A)$.
Moreover by \emph{loc.~cit.}~if $V$ is unramified at $v$, then also $\honef(E_v, T) = \honeur(E_v, T)$ and $\honef(E_v, A[\pi^n]) = \honeur(E_v, A[\pi^n])$.
\end{example}

\begin{definition}\label{def:selmer-groups}
Let $X = T, V, A, A[\pi^n]$. The Selmer group of $X$ associated to a local condition $\calF$ is defined to be the group 
\[
\hone_\calF(E, X) = \ker\bigg(\hone(E_S/E, X) \to  \bigoplus_{v \in S} \frac{\hone(E_v, X)}{\hone_{\calF_v}(E_v, X)} \; \bigg).
\]
\end{definition}

In this article we will consider mainly the Bloch--Kato Selmer groups defined in \cite{bloch-kato:l-funct-tamagawa}, usually denoted by $\honef(E, X)$, 
i.e.~the Selmer groups associated with the following local conditions: 
\begin{align*}
\honef(E_v, V) &= 
\begin{cases}
\honeur(E_v, V) \quad &\text{if $v \in S$, $v \nmid p\infty$};\\
\ker \Big(\hone(E_v, V) \to \hone(E_v, V \otimes_{\Q_p} B_\crys)\Big) \quad & \text{if $v \mid p$};\\
\hone(E_v, V) &\text{if $v \mid \infty$}
\end{cases}\\
\end{align*}
where $B_\crys$ denote the Fontaine's ring of $p$-adic crystalline periods. The group 
\[
\hone_s(E_v, X) = \hone(E_v, X)/\honef(E_v, X)
\] 
is called the singular quotient, for $v \nmid p\infty$.

If moreover the representation $V$ is ordinary at any $v \mid p$, in the sense of \cite{greenberg:iwasawa-motives}, we have for any $v \mid p$ a filtration of 
$\K[G_v]$-submodules $V = F_v^0 V \supseteq V_v^+ = F_v^1 V \supseteq F_v^2 V \supseteq \cdots$ and we may define for $X = T, V, A$ the Greenberg Selmer group 
and the strict Greenberg Selmer group. For any $v \mid p$ let $T_v^+ = V_v^+ \cap T$, $A_v^+ = V_v^+/T_v^+$ and $X_v^- = X/X_v^+$; define
\begin{align*}
\hone_\Gr(E, X) = \ker\bigg( \hone(E_S/E, X) \to \bigoplus_{v \mid p} \hone(I_v, X_v^-) \oplus \bigoplus_{\substack{v \in S \\ v \nmid p \infty}} \hone(I_v, X)\bigg),\\
\hone_\str(E, X) = \ker\bigg( \hone(E_S/E, X) \to \bigoplus_{v \mid p} \hone(G_v, X_v^-) \oplus \bigoplus_{\substack{v \in S \\v \nmid p \infty}} \hone(I_v, X)\bigg).
\end{align*}
They are related by the exact sequence
\[
0 \to \hone_\str(E, X) \to \hone_\Gr(E, X) \to \bigoplus_{v \mid p} \hone\big(G_v/I_v, \ho(I_v, X_v^-)\big). 
\]

Note that the Greenberg Selmer groups of $V$ fit into the general framework described above using the local conditions
\begin{align*}
\hone_{\calF_\Gr}(E_v, V) &= 
\begin{cases}
\honeur(E_v, V) \quad &\text{if $v \in S$, $v \nmid p\infty$};\\
\hone_\ord(E_v, V) = \ker \Big(\hone(E_v, V) \to \hone(I_v, V_v^-)\Big) \quad & \text{if $v \mid p$};\\
\hone(E_v, V) &\text{if $v \mid \infty$}
\end{cases}\\
\end{align*}
and 
\begin{align*}
\hone_{\calF_\str}(E_v, V) &= 
\begin{cases}
\honeur(E_v, V) \quad &\text{if $v \in S$, $v \nmid p\infty$};\\
\hone_\str(E_v, V) = \ker \Big(\hone(E_v, V) \to \hone(G_v, V_v^-)\Big) \quad & \text{if $v \mid p$};\\
\hone(E_v, V) &\text{if $v \mid \infty$}.
\end{cases}\\
\end{align*}
It is immediate to observe that $\hone_\Gr(E, V) = \hone_{\calF_\Gr}(E, V)$ and $\hone_\str(E, V) = \hone_{\calF_\str}(E, V)$, but in general 
$\hone_\Gr(E, X) \ne \hone_{\calF_\Gr}(E, X)$, $\hone_\str(E, X) \ne \hone_{\calF_\str}(E, X)$ for $X = T, A$. For instance we have already remarked in 
Example \ref{ex:unramified-finite} that $\honeur(E_v, A)$ contains $\honef(E_v, A)$, that is the image of $\honeur(E_v, V)$ via the natural map, 
but these are not equal in general.

However, defining for any $v \mid p$ and $X = T, V, A$,
\begin{align*}
\hone_\ord(E_v, X) &= \ker \Big(\hone(E_v, X) \to \hone(I_v, X_v^-)\Big);\\
\hone_\str(E_v, X) &= \ker \Big(\hone(E_v, X) \to \hone(G_v, X_v^-)\Big),
\end{align*}
we may write the Greenberg and strict Selmer groups in a similar fashion:
\begin{align*}
\hone_\Gr(E, X) = \ker\bigg( \hone(E_S/E, X) \to \bigoplus_{v \mid p} \frac{\hone(E_v, X)}{\hone_\ord(E_v, X_v)} \oplus \bigoplus_{\substack{v \in S \\ v \nmid p \infty}} \frac{\hone(E_v, X)}{\honeur(E_v, X)}\bigg);\\
\hone_\str(E, X) = \ker\bigg( \hone(E_S/E, X) \to \bigoplus_{v \mid p} \frac{\hone(E_v, X)}{\hone_\str(E_v, X_v)} \oplus \bigoplus_{\substack{v \in S \\ v \nmid p \infty}} \frac{\hone(E_v, X)}{\honeur(E_v, X)}\bigg).
\end{align*}
%%%%%%%%%%%%%%%%%%%%%%%%%%%%%%%%%%%%%%%%%%%%%%%%%%%%%%%%%%%%%%%%%%%%%%%

\subsection{Selmer Complexes}\label{sec:selmer-complexes}

In \cite{nekovar:selmer-complexes}, Nekovar introduced a very general theory of Selmer complexes, generalizing the above theory of Selmer groups. 
In this section we briefly recall the definition of Selmer complexes and some of the general results of \cite{nekovar:selmer-complexes} and we show how to link this very general theory to our more down to earth setting. Let $R$ be a complete Noetherian local domain of finite dimension $d$, with maximal ideal $\mfrak$
and residue field $k$. Let $S_f \subseteq S$ be the subset consisting of finite places.

The theory of Selmer complexes can be seen as a generalization of the theory of Selmer groups from modules to complexes. 
In the following $\contcohocomplex(G, M)$ will denote, for a profinite group $G$ and a continuous $R[G]$-module $M$, 
the complex of continuous cochains of $G$ with values in $M$. 
\begin{definition}\label{definition:cochain-complex-for-complexes}
Let $G$ be a profinite group and $M^\bullet$ a complex of admissible $R[G]$-modules \cite[Definition 3.2.1]{nekovar:selmer-complexes}.
The continuous cochains complex $\contcohocomplex(G, M^\bullet)$ is the total complex of the double complex $\big(\contcochains^j(G, M^i)\big)_{i, j}$ with the differentials of the 
rows induced by $\diff_{M^\bullet}^\bullet$ and the standard cohomological ones $\delta^\bullet_{M^i}$ on the columns, i.e.~its degree $n$ component is 
\[
\contcochains^n(G, M^\bullet) = \bigoplus_{i + j = n} \contcochains^j(G, M^i)
\]
and the differential $\delta_{M^\bullet}^n$ restricts to $\contcochains^j(G_v,  \diff^i_{M^\bullet}) + (-1)^j\delta^j_{M^i}$ on $\contcochains^j(G_v, M^i)$.
\end{definition}

\begin{definition}\label{def:selmer-complex}
Let $X$ a complex of $R[G_{E, S}]$-admissible modules \cite[Definition 3.2.1]{nekovar:selmer-complexes}. 
A local condition for $X$ is a collection $\Delta(X) = (\Delta_v(X))_{v \in S_f}$, where $\Delta_v(X)$ consists of a complex 
$U_v^+(X)$ of $R$-modules together with a morphism of complexes
\[
i_v^+(X) \colon U_v^+(X) \to \contcohocomplex(G_v, X).
\]
Denote $i_S^+(X) = (i_v^+(X))_{v \in S_f}$. The Selmer complex with local condition $\Delta(X)$ on $X$ is the complex
\[
\selmer(G_{E, S}, X; \Delta(X)) = \Cone\bigg( \contcohocomplex(G_{E, S}, X) \oplus \bigoplus_{v \in S_f} U_v^+(X) 
\xrightarrow{\res - i_S^+(X)} \bigoplus_{v \in S_f} \contcoho(G_{E_v}, X)\bigg)[-1],
\]
seen as an object of the derived category $\D\rmod$. Its $i$-th cohomology group is called the generalized Selmer group of $X$ with local conditions $\Delta(X)$ and denoted by $\tildecohomology^i(G_{E, S}, X; \Delta(X))$. It has an $R$-module structure.
\end{definition}

We will, however, consider only the special case where $R = \O$ is the ring of integers of a $p$-adic field $\K$, 
the complex $X$ is concentrated in degree $0$ and its $0$\,-th term is an ordinary $p$-adic Galois representation $X = T, V, A$ 
as above (in fact, these are admissible $\O[G_{E, S}]$-modules by \cite[Lemma 3.2.4]{nekovar:selmer-complexes}) and the local conditions are the following Greenberg-like ones: for $v \in S_f$ 
\[
\Delta_v(X) = 
\begin{cases}
\Delta_v^\ur(X) \colon U_v^+(X) = \contcohocomplex(G_v/I_v, X^{I_v})  \xrightarrow{\infl} \contcohocomplex(G_v, X)   & \text{if $v \nmid p$;}\\
\Delta_v^\ord(X) \colon U_v^+(X) = \contcohocomplex(G_v, X_v^+)  \rightarrow  \contcohocomplex(G_v, X)                    & \text{if $v \mid p$.}
\end{cases}
\]
Moreover we denote by $\tildecohomology^i(E, X)$ generalized Selmer groups coming from this Selmer complexes. The following proposition links them with the Selmer groups previously introduced:
\begin{proposition}[{\cite[Lemma 9.6.3]{nekovar:selmer-complexes}}]\label{prop:comparison-generalized-greenberg}
For each $X = T, V, A$, there is an exact sequence
\[
0 \to \tildecohomology^0_f(E, X) \to \cohomology^0(E, X) \to \bigoplus_{v \mid p} \cohomology^0(G_v, X) \to \honetildef(E, X) \to \hone_\str(E, X) \to 0.
\]
\end{proposition}
%%%%%%%%%%%%%%%%%%%%%%%%%%%%%%%%%%%%%%%%%%%%%%%%%%%%%%%%%%%%%%%%%%%%%%%
\subsection{Duality for Selmer complexes}\label{sec:duality-selmer-complexes}
There is a useful duality theory for Selmer complexes. First of all let us recall some definitions for complexes.  
\begin{definition}\label{def:hom-complexes}
Let $X, Y$ complexes in $\rmod$. We define the complex $\Hom_R^\bullet(X, Y)$ by 
\[
\Hom^n_R(X, Y) = \prod_{i \in \Z} \Hom_R(X^i, Y^{i+n}),
\]
with differentials defined by the formula
\[
\diff^n f = (\diff_Y^{n + i} \circ f_i + (-1)^{n -1} f_{i +1} \circ \diff_X^{i})_{i \in \Z},
\]
for $f = (f_i)_{i \in \Z} \in \Hom_R^n(X, Y)$.
\end{definition}

\begin{proposition}[{\cite[Section 2.3.2]{nekovar:selmer-complexes}}]\label{prop:dualizing-functor}
Let $I$ be an injective hull of $k$ \cite[Definition 3.2.3]{bruns-herzog:cohen-macaulay-rings} and set $J = I[n]$, i.e.~the complex concentrated in degree $-n$ 
such that $J^{-n}= I$.
The functor $D_J = \Hom_R(\, - \,, J)$ is a dualizing functor on the category of complexes of $\rmod$, meaning that the canonical morphism 
$\epsilon \colon M^\bullet \to D_J\bigl(D_J(M^\bullet)\bigr)$ is an isomorphism in $\D_\ft\rmod$ (resp.~$\D_\coft\rmod$ ), the full subcategory of $\D\rmod$
consisting of complexes with finite (resp.~cofinite) type cohomology. Note that if $M^\bullet$ is a complex of $\rmod_\ft$ (resp.~$\rmod_\coft$), 
then $D_J(M^\bullet)$ is a complex of  $\rmod_\coft$ (resp.~$\rmod_\ft$).
\end{proposition}

\begin{remark}\label{rk:inj-hull}
If $R$ is a complete discrete valuation ring with fraction field $K$, $I = K/R$ is an injective hull of $k$. 
In the case of $R = \O$ the ring of integers of a finite extension of $\Q_p$ then $D(M) := \Hom_R(M, I)$ concides with the Pontryagin dual $M^\vee$ 
for $\O$-modules of finite and cofinite type, equipped the finite type modules with the $\p$-adic topology and the cofinite type ones with the discrete topology.
So in particular this holds for the $p$-adic representations $T$ (finite type) and $A$ (cofinite type). 
\end{remark}

Fix $J = I[n]$ for some $n \in \Z$ let $X, Y$ be two complexes of admissible $R[G_{K, S}]$-modules, 
$\Delta(X), \Delta(Y)$ local conditions for $X$ and $Y$ and a morphism  $\pi \colon X \otimes_R Y \to J(1)$ of complexes of $R[G_{K, S}]$-modules.
\begin{definition}\label{def:perfect}
Suppose that $X, Y$ are bounded and either the cohomology groups of $X$ are of finite type and those of $Y$ of cofinite type or 
the converse, then we say that $\pi$ is a perfect duality if the adjunction morphism \cite[Section 1.2.6]{nekovar:selmer-complexes} 
$\adj(\pi) \colon X \to D_J(Y)(1)$ is a quasi-isomorphism.
\end{definition}

For local conditions for $\Delta(X)$, $\Delta(Y)$ there is a notion of being orthogonal complements with respect to $\pi$ \cite[Lemma\,-Definition 6.2.7]{nekovar:selmer-complexes}. This notions has an important consequence:
\begin{theorem}[{\cite[Theorem 6.3.4]{nekovar:selmer-complexes}}]\label{th:duality}
Let $X, Y, J$ as above, $\pi$ a perfect duality and $\Delta(X)$ and $\Delta(Y)$ orthogonal complements. There is a map of complexes
\[
\gamma_{\pi, h_S} \colon \selmer(X) {\longrightarrow} D_{J[-3]}(\selmer(Y))
\]
that is an isomorphism in $\D\rmod$.
\end{theorem}

Let now consider the special case of $p$-adic Galois representations: recall that, as we remarked in Section \ref{sec:p-adic-reps}, $A^\ast = D(T)$, $T^\ast = D(A)$, 
the evaluation maps
\[
\ev_2 \colon T \otimes_\O A^\ast(1) \longrightarrow (\K/\O)(1), 	\qquad \ev_1 \colon   A \otimes_\O  T^\ast(1) \longrightarrow (\K/\O)(1)
\] 
are by definition perfect dualities (in the sense of Definition \ref{def:perfect}) and the Greenberg local conditions are orthogonal complements with respect to them. Hence Theorem \ref{th:duality} holds and in this case its statement becomes the following.
\begin{theorem} \label{th:duality-rep}
There are isomorphisms
\begin{gather*}
\tildecohomology^i_f(K, T) \iso D\big(\tildecohomology^{3 - i}_f(K, A^\ast(1))\big);\\
\tildecohomology^i_f(K, A) \iso D\big(\tildecohomology^{3 - i}_f(K, T^\ast(1))\big).
\end{gather*}
\end{theorem}

\begin{proof}
We show only the first isomorphism, the proof of the latter one is analogous. We set $J = (\K/\O)[0]$. By Theorem \ref{th:duality}, we have
\[
\tildecohomology^i_f(K, T) = \cohomology^i\big(\selmer(K, T)\big)  \cong \cohomology^i\Big(D_{J[-3]}\Big(\selmer\big(K, A^\ast(1)\big)\Big)\Big). 
\]
But
\begin{align*}
&D_{J[-3]}\Big(\selmer\big(K, A^\ast(1)\big)\Big)^i = \prod_{j \in \Z} \Hom_\O\Bigl(\selmer\bigl(K, A^\ast(1)\bigr)^j, J^{ i + j - 3}\Bigr) \\
& =  \Hom_\O\Bigl(\selmer\bigl(K, A^\ast(1)\bigr)^{3 - i}, \K/\O\Bigr) = D \Bigl( \selmer\bigl(K, A^\ast(1)\bigr)^{3 - i} \Bigr)
\end{align*}
as  
\[
J^{ i + j - 3} = \begin{cases}
\K/\O \quad& \text{if $j = 3 - i$},\\
0 &\text{else}.
\end{cases}
\]
Therefore, since $\K/\O$ is injective and therefore $D = \Hom_\O( \, - \, , \K/\O)$ is exact and hence it commutes with the $i$-th cohomology functor
\begin{align*}
\tildecohomology^i(K, T) &\cong \cohomology^i\Big(D_{J[-3]}\Big(\selmer\big(K, A^\ast(1)\big)\Big)\Big) \cong D\Big(\tildecohomology^{3-i}_f\big(K, A^\ast(1)\big)\Big).\qedhere
\end{align*}
\end{proof}
%%%%%%%%%%%%%%%%%%%%%%%%%%%%%%%%%%%%%%%%%%%%%%%%%%%%%%%%%%%%%%%%%%%%%%%
\section{Anticyclotomic Iwasawa Theory} \label{sec:anticyclotomic-iwasawa-theory}

The aim of this section is to introduce Selmer groups and Selmer complexes over the anticyclotomic extension $K_\infty$ of $K$.

\subsection{Anticyclotomic extension}\label{sec:anticyclotomic-extension}

Recall that $K$ denotes the imaginary quadratic field of Section \ref{sec:notations}, and assume moreover that  
$p \nmid h_K$, where $h_K$ is the class number of $K$. Consider for any $n \ge 0$ the ring class field $K[p^{n+1}]$ of conductor $p^{n+1}$: $\Gal(K[p^{n+1}]/K)$ admits a canonical splitting
\[
\Gal(K[p^{n+1}]/K) \cong \Gal(K[p^{n+1}]/K[p]) \times \Gal(K[p]/K) \cong \Z/p^{n}\Z \times \Delta,
\]
where (see \cite[Theorem 7.24]{cox:primes-of-the-form} for details)
\[
\lvert \Delta \rvert = h_K \Bigl(p - \bigl(\tfrac{d_K}{p}\bigr)\Bigr) = h_K (p \pm 1),
\] 
and hence $\Delta$ has order prime to $p$. This splitting is compatible with the variation of $n$, therefore we have a splitting of the Galois group of $K[p^\infty] = \bigcup_n K[p^{n}]$ over $K$,
\[
\Gal(K[p^\infty]/K) \cong \Gal(K[p^\infty]/K[p]) \times \Gal(K[p]/K) \cong \Z_p \times \Delta.
\]
Hence $K_\infty = (K[p^\infty])^\Delta$ is a $\Z_p$-extension of $K$, called the anticyclotomic $\Z_p$-extension of $K$: it is the unique $\Z_p$-extension of $K$ which is pro-dihedral over $\Q$, i.e.
\[
\Gal(K_\infty/\Q) \cong \Gal(K_\infty/K) \rtimes \Gal(K/\Q) = \Gal(K_\infty/K) \rtimes \set{1, \tau_c},
\]
where the complex conjugation $\tau_c \in \Gal(K/\Q)$ acts by inversion, that is $\tau_c \, g \, \tau_c^{-1} = g^{-1}$ for any $g \in \Gal(K_\infty/K)$.

We fix moreover in the rest the following notations: let $\Gamma := \Gal(K_\infty/K) \cong \Z_p$ and let $K_n$, for any $n\ge 1$, 
be the unique subextension of $K_\infty/K$ such that $\Gamma_n = \Gal(K_n/K) \cong \Z/p^n\Z$. 
Denote by $\Lambda_n = \O[\Gamma_n]$ the group ring of $\Gamma_n$ and let  
$\Lambda = \projlim_n \Lambda_n = \O \llbracket \Gamma \rrbracket$ be the completed group algebra of $\Gamma$. 
It is well-known \cite[Proposition 5.3.5]{neu:cnf} that $\Lambda \cong \O\llbracket X\rrbracket$ 
via the map $\gamma \mapsto 1 + X$, where $\gamma$ is a topological generator of $\Gamma$. 
Write $\Gamma^n = \Gal(K_\infty/K_n) \subseteq \Gamma$.

\begin{remark}\label{rk:ramification-anticyclotomic}
For future reference let us study the ramification in $K_\infty$ when $p$ is split in $K$.
By general theory of $\Z_p$-extensions \cite[Proposition 13.2, 13.3]{washington:cyclotomic-fields} the ramification 
is concentrated at places above $p$ and one of them must ramify; moreover   
if $v \mid p$ ramifies in $K_\infty/K$, it is unramified in $K_n/K$ up to some $n$, but after that it is totally ramified in $K_\infty/K_n$. 
In our case, since $K_\infty/K$ is a normal extension and $p$ splits in $K$, 
both places $\p$ and $\bar{\p}$ of $K$ over $p$ have the same behavior in $K_\infty/K$, 
namely they are totally ramified as $K_\infty \cap K[1] = K$ (since $p \nmid h_K$ and $K_\infty$ is a $p$-extension).
\end{remark}
%%%%%%%%%%%%%%%%%%%%%%%%%%%%%%%%%%%%%%%%%%%%%%%%%%%%%%%%%%%%%%%%%%%%%%%
\subsection{Selmer groups for $K_\infty$}\label{sec:selmer-groups-iwasawa}

Let $T$ be an $\O$-adic representation, $V = T \otimes \K$ and $A = V/T$. Note that the cohomology groups $\hone(K_n, T)$ 
and $\hone(K_n, A)$ inherit an $\O$-module structure from $T$ and $A$, moreover the following well-known lemma applied to 
$G = G_K$, $H= G_{K_n}$ shows that they are endowed with a natural action of $G/H= \Gamma_n$ and thus they admit a 
$\Lambda_n$-module structure.

\begin{lemma}\label{lemma:action-quotient-group}
Let $G$ be a profinite group. If $X$ is a $G$-module and $H$ a closed normal subgroup of $G$, then $\hone(H, X)$ is endowed with a natural action of $G/H$.
\end{lemma}

Moreover we may use the restriction and corestriction maps  (see \cite[Chapter II, Section 5]{neu:cnf}) in order to link the cohomology groups over the different layers $K_n$. First of all note that the formation of the Selmer groups defined in Section \ref{sec:selmer-groups} is compatible with these maps, i.e.
\[
\res_{K_{n+1}/K_n}\bigl(\hone_{\ast}(K_n, A) \bigr) \subseteq \hone_{\ast}(K_{n+1}, A), \quad  \cores_{K_{n+1}/K_n}\bigl(\hone_{\ast}(K_{n+1}, A) \bigr) \subseteq \hone_{\ast}(K_n, A)  
\]
for $\ast = f, \Gr, \str$. We may therefore define
\[
\hone_\ast(K_\infty , A) = \injlim_{n, \, \res} \hone_\ast(K_n, A) \quad \text{and} \quad \hone_\ast(K_\infty, T) = \projlim_{n, \, \cores} \hone_\ast(K_n, T),
\]
which are naturally $\Lambda$-modules as $\hone_\ast(K_n, A)$ and $\hone_\ast(K_n, T)$ are stable subgroups with respect to the $\Gamma_n$-action of Lemma \ref{lemma:action-quotient-group} and hence sub-$\Lambda_n$-modules. In the following we will be interested in the structure as $\Lambda$-modules 
of the Pontryagin dual of $\honef(K_\infty, A)$. 
%%%%%%%%%%%%%%%%%%%%%%%%%%%%%%%%%%%%%%%%%%%%%%%%%%%%%%%%%%%%%%%%%%%%%%%
\subsection{Selmer complexes for $K_\infty$}\label{sec:selmer-complexes-iwasawa}

Let $R = \Lambda$ and $\mfrak = (\pi, \gamma - 1)$ its maximal ideal and suppose  that $K_\infty \subseteq K_S$.
Denote moreover by $\iota \colon \Lambda \to \Lambda$, the involution induced by $\gamma \mapsto \gamma^{-1}$.
Nekovar defines, for $T$ complex of admissible modules in $\rgksmod_\ft$ and $M$ complex of ind-admissible modules \cite[Definition 3.3.1]{nekovar:selmer-complexes} in $\rgksmod_\coft$ supported in $\set{\mfrak}$ (i.e.~$M^i = \bigcup_{n>0} M^i[\mfrak^n]$ for any $i$)
endowed with \emph{Greenberg local conditions}, Selmer complexes 
\[
\selmeriw(K_\infty/K, T) \qquad \selmer(K_S/K_\infty, M)
\]
such that (essentially by \cite[Proposition 8.8.6]{nekovar:selmer-complexes}), if $T$ is a complex concentrated in degree $0$, with an ordinary $\O$-adic Galois representation as its $0$-th term, again denoted by $T$ by abuse of notation, and $M$ is the complex concentrated in degree $0$, with $A=V/T$ as its $0$-th term, 
\[
\honetildef(K_S/K_\infty, M) \cong \injlim_{n, \, \res} \honetildef(K_n, A); \qquad \honetildefiw(K_\infty/K, T) \cong \projlim_{n, \, \cores} \honetildef(K_n, T).
\]
In the following we will use the simplified notations $\honetildef(K_\infty, A)$ and $\honetildef(K_\infty, T)$ for these objects.

The general theory of Selmer complexes in Iwasawa theory can be found in \cite[Chapter 8]{nekovar:selmer-complexes}. We will need only special cases for $p$-adic Galois representations of the following results. Let us start with duality theory.

\begin{remark}
Note that, since $\Lambda$ has the same residue field as $\O$, then $\K/\O$ can be used in order to define the duality functor 
of Proposition \ref{prop:dualizing-functor} for $R = \Lambda$, that we denote by  $\bar{D}$. Hence for any $\Lambda$-module $M$,
\[
\bar{D}(M) = \Hom_\Lambda(M, \K/\O) = \Hom_\O(M, \K/\O) = M^\vee
\]
as $\O$-modules. However the natural action of $\lambda \in \Lambda$ on $f \in D(M)$ is $(\lambda f)(x) = f(\lambda x)$, for any $x \in M$, 
while the natural of $\Lambda$ on $M^\vee$ is the group action of $\Gamma$, i.e.~$(\gamma \cdot f)(x) = f(\gamma^{-1}\cdot x)$, for $f \in M^\vee$, $x \in M$. 
In the following we denote $M^\vee$ together with the latter $\Lambda$-module structure by $D(M)$. For any $\Lambda$-module $N$, let $N^\iota$   
be $N$ itself as set, endowed with the action 
$\lambda \cdot_{\iota} x = \iota(\lambda) \cdot x$. With these notations we have $D(M) = \bar{D}(M)^\iota$.
\end{remark} 
By the above remark and Theorem \ref{th:duality}, we have the following result.
\begin{theorem}[{\cite[Section 8.9.6.1]{nekovar:selmer-complexes}}]\label{th:duality-iwasawa}
Let $T$ and $T^\ast(1)$ be two bounded complexes of admissible modules in $\ogksmod^\ad_{\O-\ft}$ endowed with Greenberg local structures. Define moreover the bounded complexes
\[
A = D(T^\ast(1))(1), \quad A^\ast(1) = D(T)(1),
\]
that are complexes of admissible modules in $\rgksmod_{\O-\coft}$ endowed with induced Greenberg local conditions. Let $J = \K/\O[0]$.
We have quasi-isomorphisms
\begin{align*}
\selmeriw(K_\infty/K, T) &\longiso {D}_{J[-3]}\bigg( \selmer(K_S/K_\infty, A^\ast(1))^\iota\bigg),\\
\selmer(K_S/K_\infty, A) &\longiso {D}_{J[-3]}\bigg( \selmeriw(K_\infty/K, T^\ast(1))^\iota\bigg)
\end{align*}
respectively in $\D^b_\ft\lambdamod$ and $\D^b_\coft\lambdamod$.
\end{theorem}

And in the special case in which $T$ is a $p$-adic representation, similarly to Theorem \ref{th:duality-rep}:
\begin{theorem}\label{th:duality-iwasawa-rep}
There are isomorphisms of $\Lambda$-modules:
\begin{gather*}
\tildecohomology^i_{f}(K_\infty, T) \iso \bar{D}\big(\tildecohomology^{3 - i}_f(K_\infty, A^\ast(1))\big);\\
\tildecohomology^i_{f}(K_\infty, A) \iso \bar{D}\big(\tildecohomology^{3 - i}_{f}(K_\infty, T^\ast(1))\big).
\end{gather*}
\end{theorem}

A classical tool in Iwasawa theory for elliptic curves is Mazur's control theorem: it gives the relations between the Selmer groups over $K$ 
and $K_\infty$. In our context it generalizes to the following exact control theorem.

\begin{theorem}[{\cite[Proposition 8.10.1]{nekovar:selmer-complexes}}]\label{th:exact-control-theorem-general}
Let $T$ be a bounded complex of admissible modules in $\rgksmod^\ad_{\O-\ft}$ endowed with Greenberg local structure. There is a canonical isomorphism
\[
\selmeriw(K_\infty/K, T) \overset{\mathbb{L}}{\otimes}_{\Lambda} \O \longiso \selmer(K, T)
\]
in $\D^b\omod$ inducing a homological spectral sequence
\[
E_{i, j}^2 = \cohomology_{i, \cont}\big(\Gamma, \tildecohomology^{-j}_f(K_\infty, T)\big) \allora \tildecohomology^{-i-j}_f(K, T).
\]
In particular, dualizing it, we get also a cohomological spectral sequence
\[
E^{i, j}_2 = \cohomology^i_\cont\big(\Gamma, \tildecohomology^j_f(K_\infty, A)\big) \allora \tildecohomology^{i+j}_f(K, A).
\]
\end{theorem}
In the case of a $p$-adic Galois representation it implies the following statement:
\begin{theorem}\label{th:exact-control-theorem-representations}
Suppose that $\tildecohomology_f^0(K, A) = 0$. It follows that the canonical map
\[
\tildecohomology_f^1(K, A)  \longiso \tildecohomology_f^1(K_\infty, A)^\Gamma
\]
is an isomorphism of $\O$-modules.
\end{theorem}

\begin{proof}
Consider the first quadrant cohomological spectral sequence of Theorem \ref{th:exact-control-theorem-general}. 
Note that
\[
E_\infty^{0, 0} \cong E_2^{0, 0} = \tildecohomology_f^0(K_\infty, A)^\Gamma \cong 
\frac{F^0 \tildecohomology_f^0(K, A)}{F^1 \tildecohomology_f^0(K, A)} 
= \tildecohomology_f^0(K, A) = 0.
\]
Moreover by the Theorem \ref{th:duality-iwasawa-rep}, we have
\[
0 = D\bigl( \tildecohomology_f^0(K_\infty, A)^\Gamma \bigr) \cong \tildecohomology_f^3(K_\infty, T)_\Gamma
\]
and hence $\tildecohomology_f^3(K_\infty, T) = 0$ by Nakayama's lemma and then, by taking its Pontryagin dual, we also have  $\tildecohomology_f^0(K_\infty, A) = 0$ and hence $E_2^{i, j} = 0$ for $j \le 0$. In fact $E_2^{i, j} = 0$ if $j \ne 1, 2$ 
(i.e.~the $E_2$-sheet has nonzero terms only in the horizontal half-lines $j = 1, 2$, $i \ge 0$) since we have
\[
\tildecohomology^j(K_\infty, A) \cong D\bigl(\tildecohomology_f^{3-j}(K_\infty, T) \bigr) = 0
\]
for $j \ge 3$, by Theorem \ref{th:duality-iwasawa-rep}.
In particular the five term exact sequence 
\[
0 \to E_2^{1, 0}  \to E^1 \to E_2^{0, 1} \to E_2^{2, 0} \to E^2
\]
becomes
\[
0 \to  0 \to \tildecohomology_f^{1}(K, A) \to  \tildecohomology_f^{1}(K_\infty, A)^{\Gamma} \to  0 \to \tildecohomology_f^{2}(K, A)
\]
and hence the edge morphism $\tildecohomology_f^{1}(K, A) \to \tildecohomology_f^{1}(K_\infty, A)^{\Gamma}$ is an isomorphism.  
\end{proof}
%%%%%%%%%%%%%%%%%%%%%%%%%%%%%%%%%%%%%%%%%%%%%%%%%%%%%%%%%%%%%%%%%%%%%%%
\section{Modular forms and Generalized Heegner cycles}\label{sec:mod-forms-gen-heegner-cycles}

Let $f = \sum_{n \ge 1} a_n q^n$ be a cusp form of level $\Gamma_0(N)$ and weight $k \ge 2$. 
Suppose moreover that $f$ is a normalized Hecke eigen-newform (thus $a_1 = 1$, $T(\ell)f = a_\ell f$ for any rational prime $\ell$).
Let $F$ be the finite extension of $\Q$ in $\bar{\Q}$ generated over $\Q$ by all the $a_n$'s, called the \emph{Hecke field} of $f$, 
and let $\O_F$ be its ring of integers. Let $\p$ be the prime of $F$ over $p$ induced by $i_p$ and from now on denote by $\K$ the completion inside $\bar{\Q}_p$ of $F$ at $\p$, by $\O$ its ring of integers and, with a slight abuse of notation, 
denote by $\p$ also its maximal ideal. We make the extra assumption that $f$ is $\p$-odinary, i.e.~$i_p(a_p) \in \O^\times$. 

The aim of this paragraph is to introduce a notion of Selmer group and Shafarevich--Tate group attached to such a modular form and to study them 
under some hypothesis. 
%%%%%%%%%%%%%%%%%%%%%%%%%%%%%%%%%%%%%%%%%%%%%%%%%%%%%%%%%%%%%%%%%%%%%%%
\subsection{Self-dual representation}\label{sec:selfdual-rep}

Let $V_\p$ be the (dual of the) representation of $G_\Q$ attached to $f$ by Deligne in \cite{deligne:formes-modulaire} 
and denote its associated group morphism by 
\[
\rho_{f, \p} \colon G_\Q \to \GL_2(\K);
\]
$V_\p$ is unramified outside the finite set of rational primes
$S = \set{ \ell \;\text{prime}\;  \colon \ell \nmid pN} \cup \set{\infty}$ and
\[
\Tr\big(\rho_{f, \p}(\Frob_\ell)\big) = i_p (a_\ell), \quad \det\big(\rho_{f, \p}(\Frob_\ell)\big) = \ell^{k - 1},
\]
for any prime $\ell \ne p$. Ribet \cite[Theorem 2.3, Proposition 2.2]{ribet:nebentypus} proved moreover that $\rho_{f, \p}$ is irreducible and that $\det \rho_{f, \p} = \chi_p^{k-1}$.

However we will not use directly this representation but we will use instead $V = V_\p^\ast(k/2)$ since it has a property that makes it closer to
the Tate module of an elliptic curve: it is self-dual, i.e.~$V^\ast(1)\cong V$. 
Indeed Nekovar in \cite{nekovar:kuga-sato} 
shows that $V$ has a $G_\Q$-equivariant lattice $T$ equipped with a $G_\Q$-equivariant, 
skew-symmetric, non-degenerate $\O$-linear pairing $[\,-,-\,] \colon T \times T \to \O(1)$, in particular this means that $T^\ast(1) \cong T$, 
and tensoring with $\K$, this shows that $V$ is self-dual. 

Furthermore, since $f$ is $\p$-ordinary, we may describe this representation more explicitly.
A result of Wiles \cite[Theorem 2.2.2]{wiles:lambda-adic} tells us that $V_\p$  
is $p$-ordinary (in the sense of \cite{greenberg:iwasawa-motives}): it is reducible as a representation of $G_p$, 
and more precisely ${\rho_{f, \p}}_{|G_p}$ is equivalent to a representation of the form
\[
\begin{pmatrix}
\epsilon_1 & \ast\\
0 & \epsilon_2
\end{pmatrix},
\]
where $\epsilon_2$ is the unramified character of $G_p$ (i.e.~$\epsilon_2(I_p) = \set{1}$) such that $\epsilon_2(\Frob_p) = \alpha$, for $\alpha \in \O^\times$ 
the unit root of the polynomial $X^2 - i_p(a_p)X + p^{k-1} = 0$ (which exists since $f$ is $\p$-ordinary). 
Hence $V_\p^\ast$ (whose associated group morphism is $\traspinv{{\rho_{f, \p}}}$) as a representation of $G_p$ 
is equivalent to 
\[
\begin{pmatrix}
\epsilon_1^{-1} & 0\\
\ast & \epsilon_2^{-1}
\end{pmatrix},
\]
and, conjugating then by the matrix $\big(\begin{smallmatrix} 0&1\\1&0\end{smallmatrix}\big)$, to 
\[
\begin{pmatrix}
\epsilon_2^{-1} & \ast\\
0 & \epsilon_1^{-1}
\end{pmatrix} = 
\begin{pmatrix}
\delta & \ast\\
0 & \delta^{-1}\chi_p^{1 - k}
\end{pmatrix},
\]
denoting $\epsilon_2^{-1}$ by $\delta$, since $\epsilon_1 \epsilon_2 = \det \rho_{f, \p} = \chi_p^{k - 1}$. 
Twisting by the $(k/2)$-th power of the cyclotomic character 
$\chi_p \colon G_p \to \Z_p^\times$ it follows that $V$ as a representation of $G_p$ is equivalent to 
\[
\begin{pmatrix}
\delta\chi_p^{k/2} & \ast\\
0 & \delta^{-1}\chi_p^{1 - k/2}
\end{pmatrix}.
\]
We find therefore an exact sequence of $\K[G_p]$-modules 
\[
0 \to V^+ \to V \to V^- \to 0,
\]
where $V^{\pm}$ has dimension $1$ over $\K$, $G_p$ acts on $V^+$ as $\delta \chi_p^{k/2}$ and on $V^-$ as $\delta^{-1} \chi_p^{1- k/2}$.

\begin{remark}\label{rk:charpol-frobenius-ell}
Let us remark for later reference that, since the group morphism associated to $V$ is $\chi_p^{k/2} \cdot \traspinv{{\rho_{f, \p}}}$, 
it follows that the characteristic polynomial of $\Frob_\ell$ over $T$ for any $\ell \ne p$ is
\[
\charpol(\Frob_\ell \vert T) = X^2 - \frac{i_p(a_\ell)}{\ell^{k/2-1}}X + \ell.
\]
Indeed the action of $\Frob_\ell$ is given by $\ell^{k/2} \cdot \trasp{\rho_{f, \p}(\Frob_\ell)}^{-1}$: its trace is $i_p(a_\ell)/\ell^{k/2-1}$ and its determinant is $\ell$.
\end{remark}
%%%%%%%%%%%%%%%%%%%%%%%%%%%%%%%%%%%%%%%%%%%%%%%%%%%%%%%%%%%%%%%%%%%%%%%
\subsection{Generalized Heegner cycles}\label{sec:gen-heegner-cycles}

From now on let $V$ and $T$ denote respectively the self-dual representation attached to $f$ and its lattice, as in Section \ref{sec:selfdual-rep}, and set $A = V/T$. 
The aim of this article is to study the Selmer groups, in particular the Bloch--Kato one, attached to $A$, 
since they generalize the classical notion of Selmer group of an elliptic curve to the case of modular forms.
In order to do that we need the \emph{generalized Heegner cycles}, firstly introduced by Bertolini, Darmon and Prasanna in \cite{bdp:generalized}, more precisely we will use a slight modification of them of Castella and Hsieh \cite{castella-hsieh:heegner-cycles-p-adic-l-functions}. We will briefly recall this construction.

Let $W_{k-2}$ denote the Kuga--Sato variety of level $\Gamma_1(N)$ and dimension $k-1$, 
i.e.~the canonical desingularization (see \cite[Lemme 5.4]{deligne:formes-modulaire} and \cite[Section 1.0.3]{scholl:motives-modular-forms}) of the $(k-2)$-fold self-fiber product of the universal generalized elliptic curve with $\Gamma_1(N)$-level structure. 
This is a scheme defined over $X_1(N)$ and the fiber at any non-cuspidal point $P = (E, \eta_N) \in X_1(N)$, where $E$ is an elliptic curve and $\eta_N$ a $\Gamma_1(N)$-level structure on $E$, is $E^{k-2}$.
 Fix moreover a (complex) elliptic curve $A$ with complex multiplication by $\O_K$. By the theory of complex multiplication 
(see for instance \cite[Theorem II.4.1, Theorem II.2.2(b)]{silverman:advanced}) $A$ is defined over the Hilbert class field $K[1]$ of $K$ and there is an isomorphism
\[
[\; - \;] \colon \O_K \iso \End_{K[1]}(A),
\]
normalized in such a way that $[\alpha]^\ast \omega = \alpha \omega$ for any $\omega \in \Omega^1_{A/K[1]}$, $\alpha \in \O_K$. 

\begin{definition}\label{def:generalized-kuga-sato}
The \emph{generalized Kuga--Sato variety} of level $\Gamma_1(N)$ and dimension $(2k-3)$ is the $(2k - 3)$-dimensional variety 
$X_{k-2} = W_{k-2} \times_{K[1]} A^{k-2}$ defined over $K[1]$.
\end{definition}
Composing the structural morphism of $W_{k-2}$, with the morphism $X_{k-2} \to W_{k-2}$ we get a proper morphism
$\pi \colon X_{k-2} \to X_1(N)$, whose fibers over a non-cuspidal point $P = (E, \eta_N)$ is $E^{k-2} \times A^{k-2}$.

Let $\Nfrak$ be a cyclic ideal of $\O_K$ of order $N$ (that there exists thanks to the Heegner hypothesis on $K$), and let $t_A \in A[\Nfrak](\C)$ be an $\Nfrak$-torsion point.
Consider the set
\[
\Isog(A) = \set{(\phi, A') : \text{$A'$ elliptic curve and $\phi \colon A \to A'$ is an isogneny defined over $\bar{K}$ }}/\cong,
\]
where $(\phi_1, A_1') \cong (\phi_2, A_2')$ if there is a $\bar{K}$-isomorphism $\iota \colon A_1' \to A_2'$ such that $\phi_2 = \iota \phi_1$. 
The generalized Heegner cycles are indexed over the subset $\Isog^{\Nfrak}(A)$ of $\Isog(A)$ consisting of isogenies $\phi \colon A \to A'$ 
whose kernel intersect $A[\Nfrak]$ trivially. 
A pair $(\phi, A') \in \Isog^{\Nfrak}(A)$ determines a point 
$P_\phi = (A', \phi(t_A))$ on $X_1(N)(\C)$. 
Let $\iota_{A'} \colon (A')^{k-2} \inj W_{k-2}$ be the embedding of $(A')^{k-2}$ as the fibre of $P_{A'}$ with respect to the structural morphism of $W_{k-2}$. Consider then the cycle $\Upsilon_\phi = \Graphfun(\phi)^{k-2}$ on the variety
\[
(A \times A')^{k-2} \iso (A')^{k-2} \times A^{k-2} \xhookrightarrow{\iota} W_{k-2} \times A^{k-2} = X_{k-2}.
\]
\begin{definition}\label{def:generalized-heegner-cycles}
Consider the idempotent of \cite[(2.2.1)]{bdp:generalized} 
\[
\epsilon_X \in \corr^0(X_{k-2}, X_{k-2}) \otimes \Z\Bigl[\frac{1}{2N(k-2)!}\Bigr].
\]
The generalized Heegner cycle attached to the isogeny $\phi \in \Isog^{\Nfrak}(A)$ is the cycle 
\[
\Delta_\phi = (\epsilon_X)_\ast \Upsilon_\phi = (\epsilon_X)^\ast \Upsilon_\phi \in CH^{k-1}(X_{k-2}) \otimes \Z\biggl[\frac{1}{2N(k-2)!}\biggr],
\]
supported on the fibre $X_P := \pi^{-1} (P_\phi) = (A')^{k-2} \times A^{k-2}$.
\end{definition}

\begin{proposition}[{\cite[Proposition 2.7]{bdp:generalized}}]
The cycle $\Delta_\phi$ is homologically trivial.
\end{proposition}

\begin{proof}
A detailed proof of this fact can be found in \cite{bdlp:generalized}.   
\end{proof}
In particular this means that $\Delta_\phi$ lays into the domain of the $p$-adic Abel-Jacobi map, that we are going to define in the next section.

\begin{remark}
As observed in \cite{bdp:generalized} one can deal with rationality questions about the generalized  Heegner cycles $\Delta_\phi$:  
they are always defined over some abelian extension of $K$. 
The next construction of Castella and Hsieh (see \cite[]{castella-hsieh:heegner-cycles-p-adic-l-functions}) selects a subclass of these cycles that are 
rational over the ring class fields $K[n]$ and whose properties are closer to those of the classical CM points, in particular they satisfy a sort of 
`Shimura reciprocity law' and some `norm relations' (see \cite[Lemma 4.3 and Proposition 4.4]{castella-hsieh:heegner-cycles-p-adic-l-functions}, 
these properties are crucial in order to make them into an Euler system and to use them into anticyclotomic Iwasawa theory: this last observation is the reason 
why they are better suited for us than the classical Heegner cycles of \cite{nekovar:kuga-sato}.
\end{remark}

Assume from now on that $p$ is split in $K$ and, if $k > 2$, we make the extra assumption that $d_K$ is odd or $8 \vert d_K$. 
Under these assumption we may fix a canonical elliptic curve $A$,
in the sense of Gross (see \cite[Theorem 0.1]{yang:cm-av}), which is characterized by the 
following properties:
\begin{itemize}
    \item it has CM by $\O_K$;
    \item $A(\C) = \C/\O_K$;
    \item it is a $\Q$-curve (see\cite[]{gross:arithmetic-elliptic-curves-cm}) defined over the real subfield $K[1]^+ = \Q\bigl(j(\O_K)\bigr)$ 
            of the Hilbert class field $K[1]$ of $K$;
    \item the conductor of $A$ is divisible only by prime factors of $d_K$.
\end{itemize}

Let now $c$ be a positive integer and $\mathscr{C}_c = c^{-1}\O_c/\O_K$, where $\O_c$ denotes the order of $\O_K$ with conductor $c$. Hence $\mathscr{C}_c$ 
is a cyclic subgroup of order $c$; the elliptic curve such that $A_c(\C)=A/\mathscr{C}_c$ is an elliptic curve defined over the real subfield 
$K[c]^+ = \Q\bigl(j(\O_c)\bigr)$ of the ring class field of $K$ of conductor $c$. Let $\phi_c \colon A \to A_c$ be the isogeny given by the quotient map.
Consider now a fractional $\O_c$-ideal $\afrak$ prime to $cd_K\p\Nfrak$ and the elliptic curve $A_\afrak$ such that $A_\afrak(\C) = \C/\afrak^{-1}$. 
The map $\C/c^{-1}\O_c \to \C/\afrak^{-1}$ defined by $z \mapsto cz$ corresponds to an isogeny $\lambda_\afrak \colon A_c \to A_\afrak$ and we define the isogeny 
\[
\phi_\afrak = \lambda_{\afrak} \circ \phi_c \colon A \to A_\afrak.
\]
For a suitable choice of $t_A$ (see \cite[Section 2.3]{castella-hsieh:heegner-cycles-p-adic-l-functions}, note that $t_A$ is written there with $\eta_c$), we have therefore a  
generalized Heegner cycle
\[
\Delta_\afrak := \Delta_{\phi_\afrak} \in \chow^{k-1}(X_{k-2}/K[c]) \otimes \Z\biggl[\frac{1}{2N(k-2)!}\biggr],
\]
in particular we write $\Delta_c := \Delta_{\O_c}$.
%%%%%%%%%%%%%%%%%%%%%%%%%%%%%%%%%%%%%%%%%%%%%%%%%%%%%%%%%%%%%%%%%%%%%%%
\subsection{$p$-adic Abel-Jacobi maps}

For a smooth variety $X$ over a field $E$ of characteristic $0$,  Jannsen \cite[Remark 6.15(c)]{jannsen:cont-et} defined a $p$-adic Abel-Jacobi map 
\[
\AJ^\et_X : \chow^i(X/E)_0 \to \cohomology^1\bigl(G_E, \etcohomology^{2i-1}(\bar{X}, \Z_p(i))\bigr) = \Ext^1_{G_E}\bigl(\Z_p, \etcohomology^{2i-1}(\bar{X}, \Z_p(i))\bigr),
\]
where $\bar{X} = X \otimes_E \bar{E}$, so that for any cycle $Z$ on $X$ homologous to zero, $\AJ^\et_X(Z)$ is the extension $E_Z$
obtained by the pull-back in the rightmost square of the following diagram
\[
\begin{tikzcd}
0 \ar[r] & \etcohomology^{2i-1}(\bar{X}, \Z_p(i)) \ar[r] & \etcohomology^{2i-1}(\bar{U}, \Z_p(i)) \ar[r] & \cohomology^{2i}_{\et, \lvert\bar{Z}\rvert}(\bar{X}, \Z_p(i)) \ar[r] & \etcohomology^{2i}(\bar{X}, \Z_p(i))\\
0 \ar[r] & \etcohomology^{2i-1}(\bar{X}, \Z_p(i)) \ar[r] \ar[u, equal] & E_Z \ar[r] \ar[u, hook] & \Z_p \cdot \bar{\cl}_X(\bar{Z}) \ar[u, hook] \ar[r] & 0
\end{tikzcd},
\]
where $\lvert\bar{Z}\rvert$ is the support of $\bar{Z}$, $\bar{U} = \bar{X} \setminus \lvert\bar{Z}\rvert$ and $\bar{\cl}_X$ is the composition of the $p$-adic étale cycle map \cite[Lemma 6.14 and Theorem 3.23]{jannsen:cont-et}
\[
\cl_X = \cl_X^i \colon \chow^i(X/E) \to \etcohomology^{2i}(X, \Z_p(i)),
\]
and the restriction $\etcohomology^{2i}(X, \Z_p(i)) \to \etcohomology^{2i}(\bar{X}, \Z_p(i))^{G_E}$.

In the following let $E$ be a number field containing the Hilbert class field $K[1]$ of $K$, so that $X = X_{k-2}$ is defined over $E$, and let $i=k-1$. Applying the projector $\epsilon_X$ to the diagram above and allowing coefficients of cycles in $\Z_p$, we obtain a $p$-adic Abel-Jacobi map  
\[
\AJ^\et_E \colon \chow^{k-1}(X_{k-2}/E)_0 \otimes \Z_p \to \cohomology^1\bigl(G_F, \epsilon_X \etcohomology^{2k-3}(\bar{X}_{k-2}, \Z_p(k-1))\bigr).
\]
In particular if $\Delta_\phi$ is $E$-rational we may compute $\AJ^\et_E(\Delta_\phi)$ as the extension $E_{\Delta_\phi}$ obtained 
by pull-back in the rightmost square of the following diagram, where $X_P^\flat := X_{k-2} \setminus X_P$ (observe that $\Delta_{\phi}$ is supported in $X_P$), 
\[
\begin{tikzcd}[cramped, column sep= tiny]
0 \ar[r] & \epsilon_X \etcohomology^{2k-3}(\bar{X}_{k-2}, \Z_p)(k-1) \ar[r] & E_{\Delta_{\phi}} \ar[d] \ar[r] & \Z_p \cdot \bar{cl}_{X_P}(\bar{\Delta}_\phi) \ar[d, hook] \ar[r] & 0\\
0 \ar[r] & \epsilon_X \etcohomology^{2k-3}(\bar{X}_{k-2}, \Z_p)(k-1) \ar[r] \ar[u, equal] & \epsilon_X \etcohomology^{2k-3}(\bar{X}^{\flat}_{P}, \Z_p)(k-1) \ar[r] & \epsilon_X \etcohomology^{2k-4}(\bar{X}_P, \Z_p)(k-2) \ar[r] & 0
\end{tikzcd}.
\]
Castella and Hsieh show moreover \cite[Section 4.2, 4.4]{castella-hsieh:heegner-cycles-p-adic-l-functions} that for any 
locally algebraic anticyclotomic character 
\[
\chi \colon \Gal(K[p^\infty]/K) \to \O_F^\times
\] 
of infinity type $(j, -j)$ there is a $G_K$ equivariant projection
\[
\epsilon_X \etcohomology^{2k-3}(\bar{X}_{k-2}, \Z_p(k-1)) \to T \otimes \chi;
\]
for $\chi = \triv$, tensoring with $\O$, this gives a $p$-adic Abel-Jacobi map
\[
\AJ^\et_E \colon \chow^{k-1}(X_{k-2}/E)_0  \otimes \O \to \hone(E, T).
\]
In particular we will be interested in $\AJ^\et_c := \AJ^\et_{K[c]}$, in fact the generalized Heegner cycle $\Delta_c$ is rational over $K[c]$, 
therefore we may consider its image $z_{f, c} := \AJ^\et_c(\Delta_c) \in \hone(K[c], T)$. 
These cohomology classes, that we still call generalized Heegner cycles, 
will play the role of the Heegner points on elliptic curves: they form an anticyclotomic Euler system.
%%%%%%%%%%%%%%%%%%%%%%%%%%%%%%%%%%%%%%%%%%%%%%%%%%%%%%%%%%%%%%%%%%%%%%%
\subsection{The Shafarevich--Tate group}\label{sec:shafarevich-tate-mod-forms}

In this section we will generalize to our context also the Mordell-Weil and the Shafarevich--Tate groups of an elliptic curve. 
Let $\imajt(K[c]) = \im \AJ^\et_{\O, c}$, that we will take as an analogue of the $K[c]$-rational points of an elliptic curve in the higher weight case. 
Note that  $\imajt(K[c]) \subseteq \honef(K[c], T)$ by \cite[Section II.1.4]{nekovar:p-adic-height-heegner-cycles} 
or \cite{niziol:p-adic-regulators} and hence $\imajt(K[c]) \otimes \K/\O$ injects into $\honef(K[c], A)$. Following the approach of Nekovar \cite[Section 11]{nekovar:kuga-sato} we give the following definition. 
\begin{definition}\label{def:imajt-ring-class-fields}
We define the ($\p$-part of the) Shafarevich--Tate group $\shat(f/K[c])$ of $f$ over $K[c]$ by the exact sequence 
\[
\begin{tikzcd}
0 \ar[r] & \imajt(K[c]) \otimes \K/\O \ar[r] & \honef(K[c], A) \ar[r] & \shat(f/K[c]) \ar[r] & 0.
\end{tikzcd}
\]
\end{definition}

We thank the anonimous referee for pointing it out to us that this definition is essentially equal to the one given by Bloch--Kato in \cite[(5.13)]{bloch-kato:l-funct-tamagawa} assuming a suitable standard conjecture on the bijectivity of the regulator map.

Note moreover that a similar definition could be given for $\shat(f/K)$, if we had a definition of $\imajt(K)$, but this cannot be defined as the image of the Abel-Jacobi map 
as above for $K[c]$, as the Abel-Jacobi map was defined only for a field $E$ cointaining $K[1]$. However we can give an alternative definition of it that 
suits our purpose. Set $\calG_1 = \Gal(K[1]/K)$ and consider the restriction map
\[
\res_{K[1]/K} \colon \hone(K, T) \to \hone(K[1], T)^{\Gal(K[1]/K)}.
\]
Under some standard assumptions, as those that we will assume in Definition \ref{def:admissible-primes}, this becomes an isomorphism of $\O$-modules and it restricts to an isomorphism
\[
\res_{K[1]/K} \colon \honef(K, T) \iso \honef(K[1], T)^{\calG_1}.
\]
\begin{definition}\label{def:image-abel-jacobi-K}
Assume that the restriction $\res_{K[1]/K}$ is an isomorphism, in this case define
\[
\imajt(K) := \res_{K[1]/K}^{-1}\bigl( \imajt(K[1])^{\calG_1} \bigr)  \subseteq \honef(K, T).
\] 
And we define the ($\p$-part of the) Shafarevich--Tate group $\shat(f/K)$ of $f$ over $K$ by the short exact sequence
\[
\begin{tikzcd}
0 \ar[r] & \imajt(K) \otimes \K/\O \ar[r] & \honef(K, A) \ar[r] & \shat(f/K) \ar[r] & 0.
\end{tikzcd}
\]
\end{definition}
The main reason why we give this definition is because, in order to use the Heegner cycles to give informations about $\honef(K, A)$, $\imajt(K)$ and $\shat(f/K)$,
we need that $\imajt(K)$ contains the \emph{basic generalized Heegner cycle}
\[
z_{f, K} = \cores_{K[1]/K} (z_{f, 1}).
\]
Indeed
\[
\res_{K[1]/K} (z_{f, K}) = \Tr_{K[1]/K} (z_{f, 1}) \in \imajt(K[1])^{\calG_1}.
\] 
Hence $z_{f, K} \in \imajt(K)$. 

\begin{remark}\label{rk:classical-heegner-cycles}
Several papers (see for instance \cite[]{nekovar:kuga-sato}, \cite{besser:finiteness-sha}, \cite{longo-vigni:beilinson}, \cite{masoero:sha}) 
use classical Heegner cycles instead of the generalized ones; they consider therefore a different Abel-Jacobi map, 
\[
\AJ^\et \colon \chow^{k/2}(\kugasatoplain{N}{k-2}/K)_0 \otimes \O \to \honef(K, T), 
\]
where $\kugasatoplain{N}{k-2}$ is the Kuga--Sato variety of level $\Gamma(N)$ and dimension $k-1$, 
defining $\Lambda_\p(K)$ and the Shafarevich--Tate group  $\sha_{\p^\infty}(f/K)$ respectively as the image and the cokernel of this map. 
These may, at least in principle, differ by the groups $\imajt(K)$ and $\shat(f/K)$ that we defined here. 
\end{remark} 
%%%%%%%%%%%%%%%%%%%%%%%%%%%%%%%%%%%%%%%%%%%%%%%%%%%%%%%%%%%%%%%%%%%%%%%
\section{Vanishing of $\shat(f/K)$}\label{sec:vanishing-of-sha}

In \cite[Section 7.3]{castella-hsieh:heegner-cycles-p-adic-l-functions} Castella and Hsieh use the classes $z_{f, c}$ in order to construct an 
anticyclotomic Euler system as in the following definition. Let $\squarefree$ denote the set of square-free products 
of primes $\ell$ inert in $K$ such that $\ell \nmid 2pN$. For any $\ell$ inert in $K$ let $\Frob_\ell := \Frob_\lambda$, 
the Frobenius element at $\lambda$ in $G_K$, where $\lambda$ is the unique prime of $K$ over $\ell$. 
The embedding $i_\ell$ induces a compatible sequence of primes $\lambda_n$ of $K[n]$ over $\lambda$ with $n \in \squarefree$, 
by $K[n]_\lambda$ we mean the completion of $K[n]$ at $\lambda_n$ and by $\kappa_n$ its resudue field. Let $\loc_{\ell}$ be the localization map 
\[
\loc_\ell \colon \hone(K[n], T) \to \hone(K[n]_\lambda, T).
\]
Let $w_f \in \set{\pm 1}$ be the eigenvalue of $f$ with respect to the Atkin-Lehner involution $w_N$ 
and let $\sigma_{\bar{\Nfrak}}$ be the image by the Artin reciprocity map of $\overline{\Nfrak}$, i.e.~the automorphism $\sigma_{\bar{\Nfrak}} \in \Gal(K[1]/K)$ corresponding to the class of $\overline{\Nfrak}$ in the class group of $K$ by class field theory.

\begin{definition}
An anticyclotomic Euler system for $f$ is a collection $\set{c_n}_{n \in \squarefree}$ of cohomology classes $c_n \in \hone(K[n], T)$ such that for any 
$n = m\ell \in \squarefree$:
\begin{enumerate}[label=(E\arabic*)]
\item $\cores_{K[n]/K[m]}(c_n) = a_\ell \, c_m$;
\item $\loc_\ell(c_n) = \res_{K[m]_\lambda/K[n]_\lambda}(\Frob_\ell \cdot \loc_\ell(c_m))$;
\item $\tau_c \cdot c_n = w_f ({\sigma_{\bar{\Nfrak}}} \cdot c_n)$.
\end{enumerate}
\end{definition}

Castella and Hsieh prove that the set $\set{z_{f, n}}_{n \in \squarefree}$ form an Euler system  
and they use it as an input of the Kolyvagin's method, as modified by Nekovar in \cite{nekovar:kuga-sato}, leading to the following theorem:
\begin{theorem}[{\cite[Theorem 7.7]{castella-hsieh:heegner-cycles-p-adic-l-functions}}]\label{th:selmer-group-V-castella-hsieh}
If $z_{f, K}$ is non-torsion in $\hone(K, T)$, then 
\[
\honef(K, V) = \K \cdot z_{f, K}.
\]
\end{theorem}

In fact they prove \cite[see][Theorem 7.19]{castella-hsieh:heegner-cycles-p-adic-l-functions} that under the non-torsion hypothesis on $z_{f, K}$,
there is a constant $C$ 
such that
\[
p^C \biggl( \frac{\honef(K, A)}{(\K/\O)z_{f, K}} \biggr) = 0,
\]
or, if we are in a situation such that $\res_{K[1]/K}$ is an isomorphism and hence $\shat(f/K)$ can be defined, 
that $p^C$ kills $\shat(f/K)$.
For classical Heegner cycles, \cite{besser:finiteness-sha} shows that the method of \cite{nekovar:kuga-sato} can 
be refined in some cases, in order to compute this constant $C$. We will show now that the same holds for generalized Heegner cycles.
%%%%%%%%%%%%%%%%%%%%%%%%%%%%%%%%%%%%%%%%%%%%%%%%%%%%%%%%%%%%%%%%%%%%%%%
\subsection{Non-exceptional primes}\label{sec:non-exceptional-primes}
We need first to define a set $\Psi(f)$ of \emph{exceptional primes}, that we exclude, depending of $f$, that we suppose from now on to be a 
{non-CM} modular form.
\begin{definition}\label{def:admissible-primes}
Let $\Psi(f)$ be the set of rational primes consisting of the following primes:
\begin{itemize}
\item the primes $p \mid 6N\phi(N)(k-2)!$;
\item the primes that ramify into the Hecke field of $f$;
\item the primes such that the image of $\rho_{f, p} \colon G_\Q \to \GL_2(F \otimes \Q_p)$ does not contain the set 
\[
\set{g \in \GL_2(\O_F \otimes \Z_p) : \det g \in (\Z_p^\times)^{k-1} }.		
\]
\end{itemize}
\end{definition}

\begin{remark}\label{rk:admissible-primes}
The previous definition slightly differs from \cite[Definition 6.1]{besser:finiteness-sha} as we need to exclude also $p \mid \phi(N)(k-2)!$ 
in order to get the integrality of the Abel-Jacobi map and since we want to work with the self-dual lattice $T$ coming from \cite{nekovar:kuga-sato}; 
we exclude also the prime $3$ as $\SL_2(\F_2) \cong S_3$ is solvable and therefore Lemma \ref{lemma:res-iso} does not 
hold anymore. However this last condition is not necessary in order to get Theorem \ref{th:besser}, 
but it allows us to state it in terms of $\shat(f/K)$ and we will assume it in our applications. 
\end{remark}

The hypothesis on the image of $\rho_{f, \p}$ is a kind of ``big image'' property and it is satisfied for all but a finite number of primes by 
\cite[Theorem 3.1]{ribet:l-adic-ii}. It implies many technical properties on the reduced representations $A[p^M]$ and their cohomology, as the next lemma. 
From now on we use the notations $\calG_n = \Gal(K[n]/K)$ and $G(n) = \Gal(K[n]/K[1])$.
\begin{lemma}\label{lemma:res-iso}
Let $p$ be a non-exceptional prime. Then $\ho(K[n], A[p^M])=\ho(K, A[p^M])=0$ for any $M \ge 1$. 
In particular the restriction maps 
\begin{align*}
\res_{K[n]/K[1]} \colon &\hone(K[1], A[p^M]) \longrightarrow \hone(K[n], A[p^M])^{G(n)},\\
\res_{K[1]/K} \colon &\hone(K, A[p^M]) \longrightarrow \hone(K[1], A[p^M])^{\calG_1}
\end{align*}
are isomorphisms.
\end{lemma}

\begin{proof}
It follows from \cite[Lemma 3.10]{longo-vigni:beilinson}, since $K/\Q$ and $K[n]/\Q$ are solvable extensions, and by the inflation-restriction exact sequence.
\end{proof}

\begin{remark}\label{rk:res-iso-T}
Lemma \ref{lemma:res-iso} implies in particular that $\imajt(K)$ and $\shat(f/K)$ can be defined. Indeed 
\[
\ho(K[1], T) = \projlim_m \ho(K[1], A[p^m]) = 0
\]
and therefore $\res_{K[1]/K}$ is an isomorphism by the inflation-restriction exact sequence.
\end{remark}

An other interesting consequence of Lemma \ref{lemma:res-iso} is the following corollary.
\begin{corollary}\label{cor:selmer-image-aj-free}
If $p$ is not an exceptional prime, then the $\O$-module $\hone(K, T)$ is torsion-free. In particular $\honef(K, T)$ and $\imajt(K)$ are free 
$\O$-modules of finite rank.
\end{corollary}

\begin{proof}
It is known that $\honef(K, T)$ is finitely generated of finite rank over $\O$, therefore the second statement follows immediately from the first. 
Since an $\O$-module may have only $\p$-torsion, it is enough to show that 
$\hone(K, T)[\p] = 0$. Since $p$ is unramified in $F$ by assumption and hence $p$ is a uniformizer of $\O$, therefore we have $\hone(K, T)[\p] = \hone(K, T)[p]$.

Consider now  the short exact sequence 
$
\begin{tikzcd}[cramped, sep=small]
0 \ar[r] & T \ar[r, "\cdot p"] & T \ar[r, twoheadrightarrow] & A[p] \ar[r] & 0
\end{tikzcd}
$, 
the induced long exact sequence in cohomology
\[
\begin{tikzcd}
\ar[r] & \ho(K, A[p]) \ar[r] & \hone(K, T) \ar[r, "\cdot p"] & \hone(K, T) \ar[r] & \hone(K, A[p]) & \ar[from=l]    
\end{tikzcd}
\]    
shows that $\hone(K, T)[p]$ is a quotient of $\ho(K, A[p])$ and the latter group is trivial by Lemma \ref{lemma:res-iso}.
\end{proof}

Finally we state the main result of this section.
\begin{theorem}\label{th:besser}
Let $p$ be a non-exceptional prime and $z_{f, K}$ be non-torsion in $\hone(K, T)$. Then 
\[
p^{2\calI_p} \shat(f/K) = 0,
\]
where $\calI_p$ is the smallest non-negative integer such that $z_{f, K}$ is nonzero in $\hone(K, A[p^{\calI_p + 1}])$. In particular, if $\calI_p = 0$, 
then $\shat(f/K)=0$ and $\imajt(K) \otimes \K/\O = \honef(K, A) = z_{f, K} \cdot \K/\O$.  
\end{theorem}

\begin{proof}
This is just \cite[Theorem 1.2]{besser:finiteness-sha}, where we replace the generalized Heegner cycles to the classical ones in the definition of the classes 
$P(n)$. Note that the proof of \cite[Theorem 1.2]{besser:finiteness-sha} is a formal consequence of the properties enjoyed by the classes $P(n)$ listed in 
\cite[Proposition 3.2]{besser:finiteness-sha} (see \cite[Section 6]{besser:finiteness-sha}), therefore in order to prove our result it is enough to show that these
properties still hold with the new definition of the classes $P(n)$. This will be done in Section \ref{sec:euler-system-generalized-heegner-cycles}: 
they are respectively Proposition \ref{prop:besser-property-complex-conjugation}, \ref{prop:besser-property-finite-condition} and 
\ref{prop:besser-property-finite-singular}.
\end{proof}

\begin{remark}\label{rk:identification-modulo-p-M}
The long exact sequence induced in cohomology by the short exact sequence
$
\begin{tikzcd}[cramped, column sep=small]
0 \ar[r] &  T \ar[r, "\, \cdot p^M"] &  T \ar[r] & A[p^M] \ar[r] & 0
\end{tikzcd}
$
shows that $z_{f, K} = 0$ in $\hone(K, A[p^M])$ if and only if it belongs to $p^M \hone(K, T)$, i.e.~if and only if it is divisible by $p^M$ in $\hone(K, T)$.
It follows that $\calI_p$ can be seen as the order of $z_{f, K}$, i.e.~the biggest integer $M$ such that $p^M \mid z_{f, K}$ in $\hone(K, T)$. 
In particular $\calI_p = 0$ if and only if $z_{f, K}$  is not divisible by $p$ in $\hone(K, T)$.
\end{remark}
%%%%%%%%%%%%%%%%%%%%%%%%%%%%%%%%%%%%%%%%%%%%%%%%%%%%%%%%%%%%%%%%%%%%%%%
\subsection{The Euler system of generalized Heegner cycles}\label{sec:euler-system-generalized-heegner-cycles}

Here we define the Kolyvagin classes $P(n)$ used in the proof of Theorem \ref{th:besser} and we establish their properties.
For any $M \ge 1$, we define the set $S(M)$ of \emph{$M$-admissible primes} as the set of rational primes $\ell$ such that
\begin{itemize}
    \item $\ell \nmid N p d_K$;
    \item $\ell$ is inert in $K$;
    \item $p^M \mid a_\ell$ and $p^M \mid \ell+1$;
    \item $p^{M+1} \nmid \ell + 1 \pm a_\ell$. 
\end{itemize}
Let $n$ be a square-free product of primes $\ell \in S(M)$; recall that we defined $G(n) = \Gal(K[n]/K[1])$ and $\calG_n = \Gal(K[n]/K)$. 
One proves that $G(n) \cong \prod_{\ell} G(\ell)$ and $G(\ell)$ that is cyclic of order $l+1$, 
say generated by an element $\sigma_\ell$. For any $\ell \in S(M)$ we define the operator 
\[
D_\ell = \sum_{i=1}^\ell i\sigma_\ell^i \in \Z[G(\ell)];
\] 
it satisfies the \emph{telescopic identity} $(\sigma_\ell - 1)D_\ell = \ell + 1 - \Tr_\ell$, where $\Tr_\ell = \sum_{i = 0}^\ell \sigma_\ell^i$.  
For $\ell_1 \ne \ell_2$, $D_{\ell_1}$ commutes with $D_{\ell_2}$, therefore we define the ($n$-th) Kolyvagin's derivative operator as
\[
D_n = \prod_{\ell \mid n} D_\ell \in \Z[G(n)].
\]
In the following any $D \in \Z[G(n)]$ and $\Tr_\ell \in G(\ell)$ will be seen as operators on the cohomology groups of $T$ and $A[p^M]$ 
via the standard Galois action. 

The starting point of the Kolyvagin method is the following lemma. The notation $\red_{p^M}$ will denote the map induced in cohomology by $T \surj T/p^MT \cong A[p^M]$.
\begin{lemma}\label{lemma:invariant-kolyvagin-classes}
The class $D_n\bigl(\red_{p^M}(z_{f, n})\bigr) \in \hone(K[n], A[p^M])$ is fixed by the action of $G(n)$. 
\end{lemma}

\begin{proof}
Let $n = m \cdot \ell$, then
\begin{align*}
(\sigma_\ell -1)D_n(z_{f, n}) &= (\sigma_\ell - 1)D_\ell D_m(z_{f, n})\\
&=(l+1)D_m(z_{f, n}) - \Tr_\ell D_m(z_{f, n})\\
&=(l+1)D_m(z_{f, n}) - D_m \Tr_\ell(z_{f, n})\\
&=(l+1)D_m(z_{f, n}) - a_\ell \bigl(\res_{K[n]/K[m]} D_m(z_{f, m})\bigr)
\equiv 0 \bmod p^M.
\end{align*}

Indeed:
\begin{itemize}
\item $\Tr_\ell(z_{f, n}) =  \res_{K[n]/K[m]} \circ \cores_{K[n]/K[m]}(z_{f, n}) = a_\ell\res_{K[n]/K[m]}(z_{f, m})$ by (E1);
\item $\ell +1, \, a_\ell \equiv 0 \mod p^{M}$ as $\ell \in S(M)$.
\end{itemize}
Thus, by Remark \ref{rk:identification-modulo-p-M}, the class 
$D_n\bigl(\red_{p^M}(z_{f, n})\bigr)$ is fixed by the action $G(\ell)$ for any $\ell \mid n$ and hence by the action of $G(n)$. 
\end{proof}

Therefore by Lemma \ref{lemma:res-iso} there is a unique class $\dkol{n} \in \hone(K[1], A[p^M])$ such that
\[
\res_{K[n]/K[1]} \bigl(\dkol{n}\bigr) = D_n\bigl(\red_{p^M} (z_{f, n})\bigr);
\]
set
\[
P(n) = \cores_{K[1]/K} \bigl(\dkol{n}\bigr) \in \hone(K, A[p^M]).
\]
\begin{remark}
Note that $D_1 = \id$, hence $P(1) = \red_{p^M} (z_{f, K})$.
\end{remark}

We come now to the properties of the classes $P(n)$. For any $n$ square-free product of primes of $S(M)$ we define 
$\epsilon_n = (-1)^{\omega(n)}w_f$, where $\omega(n)$ is the number of prime factors of $n$.
\begin{proposition}\label{prop:besser-property-complex-conjugation}
The class $P(n)$ belongs to the $\epsilon_n$-eigenspace of the complex conjugation $\tau_c$ 
acting on $\hone(K, A[p^{M}])$.
\end{proposition}

\begin{proof}
The proof follows essentially from the computations of \cite[Proposition 5.4(1)]{gross:kolyvagin}, replacing (E1) and (E3) to the analogous properties
for the Heegner points on an elliptic curve. We just sketch the argument for the convenience of the reader. 
Gross shows that we have the relation $\tau_c D_\ell = -\sigma_\ell D_\ell \tau_c + k \, \tau_c \Tr_\ell$, for some $k \in \Z$: 
it follows that
\begin{align*}
\tau_c \cdot (D_n z_{f, n}) &= \tau_c \bigl(\prod_{\ell \mid n} D_\ell \bigr) \cdot z_{f, n} \equiv 
                                    (-1)^{\omega(n)} \bigl(\prod_{\ell \mid n} \sigma_{\ell}\bigr) D_n\cdot \tau_c \cdot z_{f, n} \\
                                    &= (-1)^{\omega(n)} w_f \sigma_{\bar{\Nfrak}} \bigl(\prod_{\ell \mid n} \sigma_{\ell}\bigr) (D_n z_{f, n}) \equiv 
                                    \epsilon_n \sigma_{\bar{\Nfrak}} (D_n z_{f, n}) \bmod p^M,
\end{align*}
since $\Tr_\ell z_{f, n} = a_\ell \res_{K[n]/K[n/\ell]}(z_{f, m}) \equiv 0 \bmod p^M$ for $\ell \in S(M)$ and by Lemma \ref{lemma:invariant-kolyvagin-classes}.
Therefore, since $\res_{K[n]/K[1]}$ is an isomorphism by Lemma \ref{lemma:res-iso}, $\tau_c \dkol{n} = \epsilon_n \sigma_{\bar{\Nfrak}} \dkol{n}$.
Thus
\begin{align*}
\res_{K[1]/K} \bigl(\tau_c \cdot P(n)\bigr) &= \tau_c \cdot \bigl( \res_{K[1]/K} \circ \cores_{K[1]/K} \dkol{n} \bigr) =  \tau_c \cdot \bigl(\Tr_{K[1]/K} \dkol{n} \bigr) \\
    &= \Tr_{K[1]/K} \bigl(\tau_c \cdot \dkol{n}\bigr) = \Tr_{K[1]/K} \bigl( \epsilon_n \sigma_{\bar{\Nfrak}} \dkol{n} \bigr) =
    \epsilon_n (\sigma_{\bar{\Nfrak}} \Tr_{K[1]/K})  \dkol{n}\\ &=\epsilon_n \Tr_{K[1]/K} \dkol{n} = \epsilon_n \res_{K[1]/K} P(n) = 
    \res_{K[1]/K} \bigl( \epsilon_n P(n) \bigr), 
\end{align*}
since $\tau_c$ commutes with $\Tr_{K[1]/K}$, again since $K[1]$ is generalized dihedral over $\Q$, and 
since $\sigma_{\bar{\Nfrak}} \cdot \Tr_{K[1]/K} =  \Tr_{K[1]/K}$ as $\sigma_{\bar{\Nfrak}} \in \calG_1$.
Thus, as $\res_{K[1]/K}$ is an isomorphism by Lemma \ref{lemma:res-iso}, it follows that $\tau_c P(n) = \epsilon_n P(n)$.
\end{proof}

\begin{proposition}\label{prop:besser-property-finite-condition}
For any $v \nmid N n$, $\loc_v P(n) \in \honef(K_v, A[p^M])$.
\end{proposition}

\begin{proof}
Let's first consider the case $v \nmid p$. Since $T$ is unramified at $v \nmid Np$ (see Example \ref{ex:unramified-finite}), then 
\[
\honef(K[n], A[p^M]) = \honeur(K[n], A[p^M]) = \ker\bigl( \hone(K[n]_v, A[p^M]) \xlongrightarrow{\res_{K[n]^\ur/K[n]}} \hone(K[n]_v^\ur, A[p^M])\bigr) 
\]
where $K[n]_v$ is the completion of $K[n]$ at a prime $v[n]$ over $v$. 
Note that since $K[n]/K$ is unramified at $v \nmid n$, 
then $K_v^\ur = K[1]_v^\ur= K[n]_v^\ur$. Thus $\hone(K[n]_v^\ur, A[p^M]) = \hone(K[1]_v^\ur, A[p^M]) = \hone(K_v^\ur, A[p^M])$ and hence
\begin{align*}
\res_{K[1]_v^\ur/K[1]_v} \bigl( \loc_v \dkol{n} \bigr)   &= \res_{K[n]_v^\ur/K[n]_v} \circ \loc_v \bigl(\res_{K[n]/K[1]} \dkol{n}\bigr) \\
                                                    &= \res_{K[n]_v^\ur/K[n]_v} \circ \loc_v \bigl( D_n \red_{p^M}z_{f, n} \bigr)\\
                                                    &= D_n \Bigl(\res_{K[n]_v^\ur/K[n]_v} \circ \loc_v \bigl( \red_{p^M}z_{f, n} \bigr)\Bigr) = 0,
\end{align*}
since $z_{f, n} \in \imajt(K, T) \subseteq \honef(K[n], T)$. Now $\loc_v P(n) \in \honef(K_v, A[p^M])$ since
\begin{align*}
\res_{K_v^\ur/K_v} \bigl(\loc_v P(n) \bigr) &= \res_{K[1]_v^\ur/K[1]_v} \circ \loc_v (\res_{K[1]/K} \circ \cores_{K[1]/K} \dkol{n}) \\
&=\sum_{\sigma \in \calG_1} \sigma \cdot \bigl(\res_{K[1]_v^\ur/K[1]_v} \circ \loc_v \dkol{n}\bigr) = 0.
\end{align*}
Now we assume $v \mid p$. The restriction map 
\[
\res_{K[n]_v/K[1]_v} \colon \frac{\hone(K[1]_v, A[p^M])}{\honef(K[1]_v, A[p^M])} \longrightarrow \frac{\hone(K[n]_v, A[p^M])}{\honef(K[n]_v, A[p^M])}
\]
is injective by \cite[Lemma 7.5]{castella-hsieh:heegner-cycles-p-adic-l-functions}, therefore  $\loc_v\bigl( \dkol{n} \bigr) \in \honef(K[n]_v, A[p^M])$ as  
\[
\res_{K[n]_v/K[1]_v} \loc_v\bigl(\dkol{n}\bigr) = \loc_v \bigl( D_n (\red_{p^M} z_{f, n})\bigr) \in \honef(K[n]_v, A[p^M])
\]
again because $z_{f, n} \in \honef(K[n], A[p^M])$. Thus $\loc_v P(n) \in \honef(K, A[p^M])$.
\end{proof}

We now introduce the finite-to-singular isomorphism $\phi^{\fs}_\ell \colon \honef(K_\lambda, A[p^M]) \to \hones(K_\lambda, A[p^M])$ 
coming from the composition of the two isomorphisms
\begin{align*}
\alpha_\ell \colon  &\honef(K_\lambda, A[p^M]) \cong \hone(K^\ur_\lambda/K_\lambda, A[p^M]) \iso A[p^M]\\
\beta_\ell \colon   &\hones(K_\lambda, A[p^M]) \cong \hone(K^\ur_\lambda, A[p^M]) \cong \hone(K^t_\lambda/K^\ur_\lambda, A[p^M])  \iso A[p^M]
\end{align*}
given by evaluation of cocycles respectively at $\Frob_\ell$ and at $\tau_\ell$, a topological generator of
\[
\Gal(K_\lambda^t/K_\lambda^\ur) \cong \hat{\Z}' = \prod_{q \ne \ell} \Z_q \cong \hspace{-5pt} \projlim_{n \ge 1, \, \ell \nmid n} \hspace{-5pt} \Z/n\Z.
\]
\begin{remark}\label{remark:finite-singular-iso-proof}
Just for this remark, in contrast to the convention we made at the beginning of Section \ref{sec:vanishing-of-sha}, let $\Frob_\ell$ denote the Frobenius automorphism of $G_\ell$ and $\Frob_\lambda$ the Frobenius of
$G_\lambda = \Gal(\bar{\Q}_\ell/K_\lambda)$. In particular $\Frob_\lambda = \Frob_\ell^2$. 
In order to see that $\alpha_\ell$ is an isomorphism it is enough to apply \cite[B.2.8]{rubin:euler-systems} and observe that 
$\Frob_\lambda$ acts trivially on $A[p^M]$: since $\ell \in S(M)$, $\charpol(\Frob_\ell \mid T) \equiv X^2 - 1 \bmod p^M$, hence
$\Frob_\lambda = \Frob_\ell^2 = \id$, as automorphisms of $A[p^M]$. 

Note then that \cite[B.2.8]{rubin:euler-systems} still holds for $G = \Gal(K^t_\lambda/K^\ur_\lambda)$: 
the result of \cite[§ XIII.1]{serre:local-fields} used in its proof holds indeed also for $\hat{\Z}'$, since it follows from a computation for cyclic groups 
and a limit argument. Therefore similarly as above in order to see that $\beta_\ell$ is an isomorphism it is enough to observe that a lift of $\tau_\ell$ to
$I_\lambda = \Gal(\bar{\Q}_\ell/K^\ur_\lambda)$ acts trivially on $A[p^M]$, but $T$ is unramified at $\ell \in S(M)$.
\end{remark}

In the following formulas $[\ast]_s$, for $\ast \in \hone(K_\lambda, A[p^M])$, denotes the image of $\ast$ in the singular quotient $\hones(K_\lambda, A[p^M])$. 
\begin{proposition}\label{prop:besser-property-finite-singular}
Let $n = m \cdot \ell$. Then there is a $p$-adic unit $u_{\ell, n}$ such that 
\[
\bigl[\loc_{\ell} P(n)\bigr]_s = \bar{u}_{\ell, n} \phi^{\fs}_\ell \bigl(\loc_\ell P(m) \bigr),
\] 
where $\bar{u}_{\ell, n}$ is the reduction of $u_{\ell, n}$ to $(\O/p^{M}\O)^\times$. In particular 
$\loc_\ell P(m) \ne 0$ if and only if $\bigl[\loc_\ell P(n)\bigr]_s \ne 0$.
\end{proposition}

\begin{proof}
The result follows from the formula of the following lemma, since $\ell$ is $M$-admissible and hence the two coefficients there are $p$-adic units.
\end{proof}

\begin{lemma} \label{lemma:finite-singular}
If $n = m \cdot \ell$, then
\[
\biggl(  \frac{\ell+1}{p^M} + \frac{(-1)^{k/2 -1} \epsilon_n a_\ell}{p^M} \biggr) \bigl[ \loc_\ell  P(n) \bigr]_s =
\biggl(-\frac{\ell+1}{p^M}\epsilon_n - \frac{a_\ell}{p^M}\biggr) \phi^{\fs}_\ell\bigl(\loc_\ell P(m)\bigr).
\]
\end{lemma}

\begin{remark}
The proof of this formula is an application of the abstract theory of \cite[Section 9]{nekovar:kuga-sato} and the formula itself already appears in the 
literature in many versions. However we will spell out a detailed proof of it, since in the occurrences of this formula that we found in the literature 
there are misprints to correct. 
For instance the original formula of \cite[Proposition 10.2(4)]{nekovar:kuga-sato} there is a problem in the definition of $\epsilon_n$, 
firstly noticed by Besser (see \cite[proof of Proposition 3.2]{besser:finiteness-sha}), and some sign errors that are carried on by all the other versions of this formula that we found in the literature. 
In \cite{elias:kuga-sato} the right hand side is wrong because of an error in the computation of $\charpol(\Frob_\ell \mid T)$. 
In \cite{castella-hsieh:heegner-cycles-p-adic-l-functions} the term $\ell^{1-r}$ shouldn't appear (on both sides).
The significance of such a formula is, in modern terms, the fact that it allows to build, from an anticyclotomic Euler system, a Kolyvagin's one (see \cite{rubin:kolyvagin}, \cite{howard:heegner-points}).
\end{remark}
\begin{proof}
Let us first consider the case where $n = \ell \cdot 1$, i.e.~the formula
\[
\biggl(\frac{\ell+1}{p^M}-\frac{(-1)^{k/2 - 1} w_f a_\ell}{p^M} \biggr)  \bigl[ \loc_\ell  P(\ell) \bigr]_s =
\biggl(\frac{\ell+1}{p^M}w_f - \frac{a_\ell}{p^M}\biggr)\phi^{\fs}_\ell\bigl(\loc_\ell P(1)\bigr).
\]
Note that we have the following tower of field extension, where $K[1]_\lambda^+$ 
is the completion at the prime above $\ell$ of the maximal real subfield $K[1]^+$ of $K[1]$:
\[
\begin{tikzcd}[row sep = tiny, column sep = small]
                            &                                               &                                                                     & \Q_\ell^t \ar[dr, dash] \ar[ddl, dash]\\
                            &                                               &                                                                     &                                           & K[\ell]_\lambda^\ur \ar[d, dash] \ar[r, equal] & \Q_\ell^\ur \cdot K[\ell]_\lambda \\
K_\lambda^\ur \ar[r, equal] & K[1]_\lambda^\ur \ar[r, equal]                & \Q_\ell^\ur \ar[rru, dash, "\ell + 1"'{xshift=-5pt}] \ar[d, dash]   &                                           & K[\ell]_\lambda\\
                            &                                               & K[1]_\lambda \ar[d, equal] \ar[rru, dash, "\ell + 1"'{xshift=-5pt}] &\\
                            &  K[1]_\lambda^+ \ar[dr, equals] \ar[ur, dash] & K_\lambda \ar[d, dash]\\
                            &                                               & \Q_\ell
\end{tikzcd}
\]
Indeed
\begin{itemize}
\item $\ell$ is inert (hence unramified) in $K$ and hence $K_\lambda^\ur = \Q_\ell^\ur$;
\item $\ell$ is totally split in $K[1]/K$, hence $K[1]_\lambda = K_\lambda$;
\item $\lambda_1$ is totally ramified in $K[\ell]/K[1]$ (that is a cyclic extension of order $\ell + 1$), i.e.~$\lambda_1 = \lambda_\ell^{\ell + 1}$: it follows that 
$[K[\ell]_\lambda : K[1]_\lambda] = \ell +1$ and $\Gal(K[\ell]_\lambda/K[1]_\lambda) \cong \Gal(\kappa_\ell/\kappa_1)$,  
that is the Galois group of a cyclic extension of finite fields, 
moreover the ramification index of $K[\ell]_\lambda/\Q_\ell$ is prime to $\ell$, thus $K[\ell]^\ur \subseteq \Q_\ell^t$;
\item $K_\lambda \subseteq K[1]_\lambda \subseteq \Q_\ell^\ur$, therefore $K_\lambda^\ur = K[1]_\lambda^\ur = \Q_\ell^\ur$;
\item $K[\ell]_\lambda^\ur = \Q_\ell^\ur \cdot K[\ell]_\lambda$;
\item $K[\ell]_\lambda^\ur/\Q_\ell^\ur$ is a cyclic extension of order $\ell + 1$, since 
\[
\begin{aligned}
    \Gal(K[\ell]_\lambda^\ur/\Q_\ell^\ur) &= \Gal(\Q_\ell^\ur \cdot K[\ell]_\lambda/\Q_\ell^\ur) \\
&\cong \Gal(K[\ell]_\lambda/\Q_\ell^\ur \cap K[\ell]_\lambda) = \Gal(K[\ell]_\lambda/K[1]_\lambda);
\end{aligned}
\]
\item $\Q_\ell \subseteq K[1]_\lambda^+ \subseteq K[1]_\lambda = K_\lambda$, but the degree of $K[1]_\lambda/K[1]_\lambda^+$ is at least $2$, 
thus $K[1]_\lambda^+ = \Q_\ell$.
\end{itemize}
Now consider the inclusions
\[
\begin{tikzcd}
\tilde{G} := \Gal(\bar{\Q}/K[1]^+) \ar[r, "\supseteq" description, phantom] \ar[d, "\rotsupseteq" description, phantom] & G :=  \Gal(\bar{\Q}/K[1]) \ar[r, "\supseteq" description, phantom] \ar[d, "\rotsupseteq" description, phantom] & H:= \Gal(\bar{\Q}/K[\ell]) \ar[d, "\rotsupseteq" description, phantom]\\
\tilde{G}_0 := \Gal(\bar{\Q}_\ell/\Q_\ell) \ar[r, "\supseteq" description, phantom] & G_0 :=  \Gal(\bar{\Q}_\ell/K[1]_\lambda) \ar[r, "\supseteq" description, phantom] & H_0:= \Gal(\bar{\Q}_\ell/K[\ell]_\lambda)
\end{tikzcd}
\]
where the squares are cocartesian, i.e.~$G_0 = \tilde{G}_0 \cap G$, $H_0 = \tilde{G}_0 \cap H$, moreover we have that 
$G/H = \Gal(K[\ell]/K[1]) = \langle \sigma_\ell \rangle$ and $G_0/H_0 = \Gal(K[\ell]_\lambda/K[1]_\lambda) = \langle \sigma_{\ell, 0} \rangle$, 
where $(\sigma_{\ell, 0})\vert_{K[\ell]} = \sigma_\ell$. 
Write $\sigma = \sigma_\ell$ and $\sigma_0 = \sigma_{\ell, 0}$. 
It is well-known that
\[
\Gal(\Q_\ell^t/\Q_\ell) = \Gal(\Q_\ell^t/\Q_\ell^\ur) \rtimes \Gal(\Q_\ell^\ur/\Q_\ell) \cong \hat{\Z}'(1) \rtimes \hat{\Z},
\]
where $\hat{\Z}'(1) = \bigl(\prod_{q \ne \ell} \Z_q\bigr)(1)$: 
$\hat{\Z}$ and $\hat{\Z}'(1)$ are procyclic and have generators $\phi$ and $\tau$ respectively such that $\phi \tau \phi^{-1} = \tau^\ell$.
It follows an analogous description in a compatible way of the subgroups
\begin{align*}
\Gal(\Q_\ell^t/K_\lambda) = \Gal(\Q_\ell^t/\Q_\ell^\ur) \rtimes \Gal(\Q_\ell^\ur/K_\lambda)  &\cong \langle \tau \rangle  \rtimes \langle \phi^2 \rangle = \hat{\Z}'(1) \rtimes 2\hat{\Z}, \\
\Gal(\Q_\ell^t/K[\ell]_\lambda) = \Gal(\Q_\ell^t/K[\ell]_\lambda^\ur) \rtimes \Gal(K[\ell]_\lambda^\ur/K[\ell]_\lambda)  &\cong \langle \tau^{\ell+1} \rangle  \rtimes \langle \phi^2 \rangle = (\ell+1)\hat{\Z}'(1) \rtimes 2\hat{\Z}.
\end{align*}
Let us denote by $\pi$ the natural projection 
\[
\pi \colon \tilde{G}_0 = \Gal(\bar{\Q}_\ell/\Q_\ell) \surj \Gal(\Q_\ell^t/\Q_\ell) \cong \Z'(1) \rtimes \hat{\Z};
\]
the induced projections $\pi \colon G_0 \surj \hat{\Z}'(1) \rtimes 2\hat{\Z}$ and $\pi \colon H_0 \surj (\ell + 1)\hat{\Z}'(1) \rtimes 2\hat{\Z}$
have the same kernel. Hence $G_0/H_0 \cong  \pi(G_0)/\pi(H_0) \cong \hat{\Z}'(1)/(\ell + 1)\hat{\Z}'(1)$,
so that we may assume that $\sigma_\ell$ is such that $\tau \bmod (\ell + 1)\hat{\Z}'(1)$ corresponds to $\sigma_{\ell, 0}$ in this isomorphism.

By \cite[Lemma 4.1]{nekovar:kuga-sato} and \cite[B.2.8]{rubin:euler-systems}, we have the following commutative diagram:
\[
\begin{tikzcd}
\hone(G_0, T) = \hone(K_\lambda, T)   & \hone(K_\lambda^t/K_\lambda, T) \ar[l, "\infl", "\sim"'] \ar[d, "\rotdxsim"]   & \hone(K_\lambda^\ur/K_\lambda, T) \ar[l,    "\infl", "\sim"'{xshift=3pt}] \ar[r, "\sim"]  \ar[d, "\rotdxsim"] & T/(\phi^2 -1)T\\
                & \hone(\hat{\Z}'(1) \rtimes 2\hat{\Z}, T)                                 & \hone(2\hat{\Z}, T)  \ar[l, "\sim"'{xshift=-3pt}] \ar[ur, "\rotatebox{20}{$\sim$}"{yshift = -3pt}]                             
\end{tikzcd}
\]
where the last isomorphism is the evaluation at $\phi^2$ and similarly $\hone(H_0, T) \cong T/(\phi^2 -1)T$. 
This means that for a $1$-cocycle $F \in Z^1(\hat{\Z}'(1)\rtimes 2\hat{\Z}, T)$, we have
\[
F(\tau^u \phi^{2v}) = (1 + \phi^2 + \dots + \phi^{2(v-1)})a + (\phi^2 - 1)b,
\]
for $a, b \in T$ and $a \equiv F(\phi^2) \bmod (\phi^2-1)T$. Indeed, $F(\tau) = 0$, as the inflation is an isomorphism and $\tau$ acts trivially on $T$ (see Remark \ref{remark:finite-singular-iso-proof}), and hence
\[
F(\tau^v) = F(\tau) + \tau \cdot F(\tau^{v-1}) = F(\tau^{v-1}) = \dots = F(\tau) = 0,
\] 
and 
\begin{align*}{}
F(\tau^u \phi^{2v}) &= F(\tau^u) + \tau^u \cdot F(\phi^{2v}) = F(\phi^{2}\phi^{2(v-1)}) = F(\phi^2) + \phi^2 \cdot F(\phi^{2(v-1)})\\
                    &\equiv a + \phi^2 \cdot (F(\phi^2) + \phi^2 \cdot F(\phi^{2(v-2)})) \equiv a + \phi^2 \cdot a + \phi^4 \cdot F(\phi^{2(v-2)})\\
                    &= \dots \equiv (1 + \phi^2 + \dots + \phi^{2(v-1)})a \bmod (1 + \phi^2)T.
\end{align*}
Now let $x := z_{f, 1} \in \hone(K[1], T) = \hone(G, T)$ and $y := z_{f, \ell} \in \hone(K[\ell], T) = \hone(H, T)$, so that $\cores^G_H(y) = a_\ell x$ and 
$z := \dkol{\ell} \in \hone(K[1], A[p^M])= \hone(G, A[p^M])$. In particular we have that $\res^G_H(z) = D_\ell(\red_{p^M}(y)) \in \hone(H, A[p^M])$. 
Note that for any $t \in T$:
\[
\sum_{i=1}^\ell i \sigma_0^i \cdot t = \sum_{i=1}^\ell i t = \biggl(\, \frac{(\ell+1)\ell}{2}\biggr) t \equiv 0 \bmod p^M T
\]
as $\sigma_0$ correspond to $\tau$ and hence acts trivially on $T$. Therefore 
\[
\res^G_{H_0}(z) = \res^H_{H_0} \bigl( \res^G_H(z)\bigr) = \res^H_{H_0} \bigl(D_\ell(\red_{p^M}(y))\bigr) = \red_{p^M} \biggl(\sum_{i=1}^\ell i \sigma_0^i \cdot \res^H_{H_0}(y)\biggr) = 0.
\]
Thus there is 
%\[
$z_0 \in \hone(G_0/H_0, A[p^M]) \cong \Hom_\cont(\langle \sigma_0  \rangle, A[p^M])$, by the inflation-restriction exact sequence, 
%\] 
such that $\infl_{G_0/H_0}^{G_0}(z_0) = \res^G_{G_0}(z) \in \hone(G_0, A[p^M])$.
Our next goal will be to calculate $z_0(\sigma_0)$. In order to do that 
we need to perform some computations at the levels of cocycles: let $\tilde{x} \in Z^1(G, T)$, $\tilde{y} \in Z^1(H, T)$ representing respectively $x$ and $y$. 

Since $\cores^G_H(y) = a_\ell x$, then $\cores^G_H(\tilde{y})- a_\ell \tilde{x}$ is a coboundary, i.e.~there is an element $a \in T$ 
(and the computations of \cite[Section 7]{nekovar:kuga-sato} show that $z(\sigma_0) \equiv -\sigma_0 a = -a \bmod p^M$), such that for any $g \in G$,
\[
\cores^G_H(\tilde{y})(g) - a_\ell \tilde{x}(g) = (g-1)a.
\]
Fixing a lift $\tilde{\sigma}_0 \in G_0$ of $\sigma_0 \in G_0/H_0$, and applying to $g = g_0 \in H_0$, becomes
\[
\sum_{i = 0}^\ell \tilde{y}(\tilde{\sigma}_0^{-i} g_0 \tilde{\sigma}^i_0) - a_\ell \tilde{x}(g_0) = (g_0 - 1)a.
\]
On the other hand, if $\pi(g_0) = \tau^u \phi^{2}$, we showed above that
\[
\tilde{x}(g_0) = a_x + (\phi^2 -1)b_x, \qquad \tilde{y}(g_0) = a_y + (\phi^2 -1)b_y,
\]
where $a_x, a_y, b_x, b_y \in T$ and $[\res^G_{G_0}(\tilde{x})]$, $[\res^H_{H_0}(\tilde{y})]$ 
correspond respectively to $a_x, a_y \bmod (\phi^2 -1)T$ and therefore evaluating at $g_0 = \phi^2$ 
(as $\tilde{y}(\tilde{\sigma}_0^{-i} g_0 \tilde{\sigma}^i_0) = \tilde{y}(\tau^{-i}\phi^2\tau^i) = \tilde{y}(\tau^{i\ell^2}\phi^2)$) we get
\[
(\ell + 1)a_y - a_\ell a_x = (\phi^2 - 1)(a + a_\ell b_x - (\ell+1)b_y)
\]
and since $T$ is torsion free, $p^M \mid a_\ell$ and $p^M \mid \ell + 1$, we may write 
\[
\frac{\ell + 1}{p^M}a_y - \frac{a_\ell}{p^M}a_x = \frac{(\phi^2 - 1)}{p^M}(a + p^M \cdot \ast),
\]
for $\ast = \frac{a_\ell}{p^M}b_x - \frac{\ell + 1}{p^M}b_y \in T$.
Now observe that $\res^G_{G_0}(x) = \loc_\ell(z_{f, 1})$ and $\res^H_{H_0}(y)=\loc_\ell(z_{f, \ell})$ and $\phi$ is the local Frobenius, therefore by (E2) 
we get that $\phi(a_x) \equiv a_y \bmod (\phi^2 - 1)T$: we may safely suppose to have previously choosen $a_y = \phi(a_x)$. Moreover on $T$
\[
\phi^2 - \frac{a_\ell}{\ell^{k/2 - 1}} \phi + \ell = 0
\]
since $\charpol(\Frob_\ell \vert T) = X^2 - a_\ell/\ell^{k/2-1} \, X + \ell$ by Remark \ref{rk:charpol-frobenius-ell} and therefore the above formula becomes
\[
\biggl(\frac{\ell + 1}{p^M} \phi - \frac{a_\ell}{p^M}\biggr)a_x = \biggl(\frac{a_\ell}{p^M} \ell^{1 - k/2}\phi - \frac{(\ell + 1)}{p^M}\biggr)a + p^M \cdot \ast
\]
and hence, reducing to $A[p^M]$:
\[
\biggl(\frac{\ell + 1}{p^M} \phi - \frac{a_\ell}{p^M}\biggr)\red_{p^M}(a_x) = \biggl(\frac{a_\ell}{p^M} (-1)^{1 - k/2}\phi - \frac{(\ell + 1)}{p^M}\biggr)\red_{p^M}(a),
\]
since  $\ell \equiv -1 \bmod p^M$. 
Now we want to express $\red_{p^M}(a)$ in terms of $z$: note that
$[\res^G_{G_0}(z)]_s = [\infl_{G_0/H_0}^{G_0}(z_0)]_s = -\beta_\ell^{-1}\bigl(\red_{p^M}(a)\bigr)$, 
as $\sigma_0$ may be lifted to $\tau$ and $\red_{p^M}(a) = -z_0(\sigma_0)$. Moreover $\red_{p^M}(a_x) = \alpha_\ell\bigl(\res^G_{G_0}(x)\bigr)$ and hence 
applying $\beta_\ell^{-1}$ the formula becomes
\[
\biggl(-\frac{\ell + 1}{p^M} \phi - \frac{a_\ell}{p^M}\biggr)\phi^{\fs}_\ell\bigl(\res^G_{G_0}(x)\bigr) = \biggl(\frac{a_\ell}{p^M} (-1)^{1- k/2}\phi+ \frac{(\ell + 1)}{p^M}\biggr)[\res^G_{G_0}(z)]_s,
\]
since $\beta_\ell \circ \phi = - \phi \circ \beta_\ell$: for $\xi \in \hone(K_\lambda^t/K_\lambda^\ur, A[p^M]) = \Hom(\Gal(K_\lambda^t/K_\lambda^\ur), A[p^M])$, 
\[
(\phi\xi)(\tau_\ell) = \phi \cdot \xi(\phi \tau_\ell \phi^{-1}) = \phi \cdot \xi(\tau_\ell^{\ell}) = \ell \bigl(\phi \cdot  \xi(\tau_\ell)\bigr) \equiv - \phi \cdot \xi(\tau_\ell) \mod p^M.
\]
Applying the corestriction, $\phi^{\fs}_\ell\bigl(\res^G_{G_0}(x)\bigr)$ and $[\res^G_{G_0}(z)]_s$ become respectively $\phi_\ell^\fs\bigl(P(1)\bigr)$ 
and $[\loc_\ell P(\ell)]_s$. Finally note that $\phi$ and $\alpha_\ell$ commute, hence $\phi_\ell^\fs \circ \phi = - \phi \circ \phi_\ell^\fs$, observe that $\phi$ acts as $\tau_c$ on $\hone(K_\lambda, A[p^M])$, as they coincide in $\Gal(K_\lambda/\Q_\ell)$, and recall that $\tau_c \cdot P(1) = w_f P(1)$ and $\tau_c \cdot P(\ell) = -w_f P(\ell)$. Therefore we obtain the desired formula.

The general formula, for $n = m \cdot \ell$, is proven in the same way if we put $G=\Gal(\bar{\Q}/K[m])$, 
$H=\Gal(\bar{\Q}/K[n])$, $G_0 = \Gal(\bar{\Q}_\ell/K[m]_\lambda)$, $H_0 = \Gal(\bar{\Q}_\ell/K[n]_\lambda)$,
$x = D_m z_{f, m}$, $y = D_m z_{f, n}$, $z = \dkol{n}$, recalling that 
$\tau_c \cdot P(n) = \epsilon_n P(n)$ and $\tau_c \cdot P(m) = \epsilon_{m}P(m)=-\epsilon_n P(m)$.
\end{proof}
%%%%%%%%%%%%%%%%%%%%%%%%%%%%%%%%%%%%%%%%%%%%%%%%%%%%%%%%%%%%%%%%%%%%%%%
\section{Consequences for Anticyclotomic Iwasawa theory}\label{sec:consequences}
In this section we prove our main result Theorem \ref{th:main} following \cite{nekovar:kolyvagin}, 
showing, under some conditions on the basic generalized Heegner cycle, that the ($\p$-part of the) Shafarevich--Tate group $\shat(f/K_\infty)$ 
(that will be defined in Definition \ref{def:shat-K-infty})
of a modular form $f$, over the anticyclotomic extension $K_\infty$ of an imaginary quadratic field $K$, vanishes.

\subsection{Setup} \label{sec:framework}

All the notations of the previous sections are in force, in particular recall from Section \ref{sec:mod-forms-gen-heegner-cycles} that 
$f = \sum_{n \ge 1} a_n q^n$ is a $\p$-odinary cuspidal newform of level $\Gamma_0(N)$ and weight $k > 2$, $V$ is the self-dual representation of
Section \ref{sec:selfdual-rep}, $K_\infty$ denotes the anticyclotomic $\Z_p$-extension of $K$ introduced in Section \ref{sec:anticyclotomic-extension} and $K_n$ its $n$-th layer. We will need moreover some technical assumptions. Recall that $F = \Q[\set{a_n}_{n>0}]$ denotes the Hecke field of $f$.

\begin{assumption}\label{assumption:1}
Take the following hypothesis on the prime $p$.
\begin{enumerate}[label=(\emph{\roman*})]  
\item $p \nmid 6N\phi(N)(k-2)!$, where $\phi$ is the Euler function;
\item the image of $\rho_{f, p} \colon G_\Q \to \GL_2(F \otimes \Q_p)$ contains the set 
\[
\set{g \in \GL_2(\O_F \otimes \Z_p) : \det g \in (\Z_p^\times)^{k-1} };		
\]
\item $p$ does not ramify in $F$;
\item $p$ splits in $K$; 
\item $p \nmid h_K$, where $h_K$ is the class number of $K$;
\item $p \nmid c_f = \big[\O_F : \O_f \big]$, where $\O_f = \Z[\set{a_n}_{n > 0}]$.
\end{enumerate}
\end{assumption}

\begin{remark}
Note that the first three hypothesis say that $p$ is a non-exceptional prime as in Definition \ref{def:admissible-primes}. 
For a discussion of the significance of these hypothesis see Remark \ref{rk:admissible-primes}; 
(\emph{iv}) is a technical hypothesis in the construction of the generalized Heegner cycles;
(\emph{v}) is important for the properties of $K_\infty$, as described in Section \ref{sec:anticyclotomic-extension}; 
(\emph{vi}) is a technical hypothesis coming from 
\cite{longo-vigni:generalized}: assuming it $\O_f \otimes \Z_p = \O_F \otimes \Z_p$. 
As observed in \cite[Remark 2.2]{longo-vigni:generalized}, once $f$ and $K$ are given, the restrictions of 
Assumption \ref{assumption:1} are satisfied by an infinite set of primes. 
\end{remark}
 	
For any $n \ge 0$ define $\calX_n = \honef(K_n, A)^\vee$, that is naturally a $\Lambda_n$-module, 
in particular we will write $\calX = \calX_0 =  \honef(K, A)^\vee$ (that is a module over $\O = \Lambda_0$). Their projective limit  
\[
\calX_\infty = \projlim_n \calX_n = \honef(K_\infty, A)^\vee
\]
has a structure of $\Lambda$-module. In the framework of Assumption \ref{assumption:1}, 
Longo and Vigni constructed in \cite{longo-vigni:generalized} the $\Lambda$-adic anticyclotomic Kolyvagin system of generalized Heegner cycles and 
they use it in order to describe the structure of $\calX_\infty$ as a $\Lambda$-module. 
\begin{theorem}[{\cite[Theorem 1.1]{longo-vigni:generalized}}]\label{th:longo-vigni}
There is a finitely generated torsion $\Lambda$-module $M$ such that $\calX_\infty$ is pseudo-isomorphic to $\Lambda \oplus M \oplus M$, i.e.~there exists a 
morphism $\eta \colon \calX_\infty \to \Lambda \oplus M \oplus M$ of $\Lambda$-modules with finite kernel and cokernel. 
\end{theorem}

Our goal is to refine this result, showing that, if the basic generalized Heegner cycle $z_{f, K}$  is non-torsion and not divisible by $p$ in $\honef(K, T)$, 
then $\calX_\infty$ is in fact free of rank one over $\Lambda$: this is the content Theorem \ref{th:main}. We will obtain it following   
the method of \cite{nekovar:kolyvagin}, i.e.~putting together the result of the Euler system argument of Section \ref{sec:vanishing-of-sha} 
with an abstract Iwasawa theoretic one.

We need moreover some other technical assumptions. From Example \ref{ex:unramified-finite}, we recall that $c_v(A) = 1$ for $v \nmid p$, since $V$ is unramified at $v \nmid pN$. We need to assume that this is the case also for the others.
\begin{assumption}\label{assumption:2}
Assume that $c_{v}(A) = 1$ for any place $v$ of $K$ such that $v \mid N$.  
\end{assumption}
Furthermore, we take the following assumption on the $q$-expansion of $f$.
\begin{assumption}\label{assumption:3}
$i_p(a_p) \not \equiv 1 \mod \p$.
\end{assumption}
%%%%%%%%%%%%%%%%%%%%%%%%%%%%%%%%%%%%%%%%%%%%%%%%%%%%%%%%%%%%%%%%%%%%%%%%%%%%%%%%%%%%%%%
\subsection{Comparison of Selmer groups, Exact Control Theorem}\label{sec:comparison-selmer-vs-generalized}

We now make a comparison of the Selmer groups introduced in Section \ref{sec:selmer-groups-and-complexes}.
In Section \ref{sec:selfdual-rep} we saw that there is an exact sequence
\[
0 \to V^+ \to V \to V^- \to 0
\]
of $\K[G_p]$-modules such that $V^{\pm}$ has dimension $1$ and $G_p$ acts on $V^+$ by $\delta \chi_p^{k/2}$ and on $V^-$ by $\delta^{-1}\chi_p^{1-k/2}$, 
for $\delta$ an unramified character and $\chi_p$ the $p$-adic cyclotomic one. Set $T^+ = T \cap V^+$, $A^+ = V^+/T^+$ and $X^- = X/X^+$, for $X = T, A$.

\begin{lemma}\label{lemma_ho(Ev, V-)}
For any a number field $E$ and any place $v \mid p$ we have $\ho(E_v, V^-) = 0$.
\end{lemma}

\begin{proof}
The inertia subgroup $I_p$ acts on the one dimensional vector space $V^-$ as $\chi_p^{1- k/2}$. Since $k>2$, there exists a $\sigma \in I_p$ such that  $\chi_p(\sigma)^{1 - k/2 } \ne 1$. This completes the proof.
\end{proof}

\begin{proposition}
$\honetildef(E, V) = \honef(E ,V) = \hone_\str(E, V) = \hone_\Gr(E, V)$.
\end{proposition}

\begin{proof}
The first equality follows from the exact sequence \cite[Proposition 12.5.9.2(\emph{iii})]{nekovar:selmer-complexes} 
for the representation attached to an ordinary modular form
\[
0 \to \bigoplus_{v \mid p} \ho(E_v, V^-) \to \honetildef(E, V) \to \honef(E ,V) \to 0
\]
and the previous lemma. The second equality follows from the exact sequence 
\[
0 \to \tildecohomology^0_f(E, V) \to \ho(E, V) \to \bigoplus_{v \mid p} \ho(E_v, V^-) \to \honetildef(E, V) \to \hone_\str(E, V) \to 0
\]
of Proposition \ref{prop:comparison-generalized-greenberg}, combined with  Lemma \ref{lemma_ho(Ev, V-)}. 
For the equality $\honef(E ,V) = \hone_\Gr(E, V)$ it is enough to prove that $\honef(E_v, V) = \ker\big(\hone(E_v, V) \to \hone(I_v, V^-)\big)$ 
for any place $v \mid p$ of $E$. This is shown in \cite[Section 3.3.3]{longo-vigni:control-theorems}.
\end{proof}
If we restrict to $E = K$ or $E = K_n$ we may compare also the Selmer groups of $A$. 
The embedding $i_p$ induces a compatible sequence of places $v_n$ of $K_n$ over $p$, let $K_{n, v}$ be the completion of $K_n$ at $v_n$. 
By Remark \ref{rk:ramification-anticyclotomic}, the inertia degree of $K_{v, n}/\Q_p$ is $1$ and hence the residue field of $K_{v, n}$ is $\F_p$. 
Hence if $I_{n, v} = I_p \cap G_{n, v}$ denote the inertia subgroup of $G_{n, v}$,
note that $G_{n, v}/I_{n, v}$ is cyclic generated by $\Frob_p \bmod \, I_{n, v}$.

\begin{lemma}\label{lemma:hone(ho(Inv, A-))}
For any place $v_n \mid p$, $\hone\big(G_{n, v}/I_{n, v}, \ho(I_{n, v}, A^-)\big) = 0$.
\end{lemma}

\begin{proof}
Since $\ho(I_{n, v}, A^-)$ is a subgroup of $A^- \cong \K/\O$, and so torsion and discrete, we may apply \cite[B.2.8]{rubin:euler-systems}. Therefore,
\[
\hone\big(G_{n, v}/I_{n, v}, \ho(I_{n, v}, A^-)\big) \cong \frac{\ho(I_{n, v}, A^-)}{(\Frob_p -1)\ho(I_{n, v}, A^-)} 
= \frac{\ho(I_{n, v}, A^-)}{(\alpha -1)\ho(I_{n, v}, A^-)},
\] 
where $\alpha$ be the $\p$-adic unit root of the polynomial $X^2 - i_p(a_p) X + p^{k-1}$ defined in Section \ref{sec:selfdual-rep}. Indeed we may see $A^-$ as $\K/\O$ together with the action given by $\delta^{-1}\chi_p^{1 - k/2}$ and hence
\[
\Frob_p \cdot \, x = \delta(\Frob_p)^{-1} \chi_p(\Frob_p)^{(1 - k/2)}  x = \alpha x
\]
for any $x \in \ho(I_{n, v}, A^{-})$ (seen as a subgroup of $\K/\O$), as $\delta(\Frob_p) = \alpha^{-1}$ and $\chi_p(\Frob_p) = 1$ 
for the choice of $\Frob_p$ made in Section \ref{sec:notations}. 
Then let $\beta$ be the non-unit (then $\beta \equiv 0 \hspace{-4pt} \mod \p$) root of the polynomial, $\alpha + \beta = i_p(a_p)$ and hence
$\alpha \equiv i_p(a_p)  \bmod \p$.
By Assumption \ref{assumption:3}, we have $\alpha \not \equiv 1 \hspace{-4pt}  \mod \p$, hence $\alpha -1$ is a $\p$-adic unit and so
$(\alpha -1)\ho(I_{n, v}, A^-) =  \ho(I_{n, v}, A^-)$.
\end{proof}

\begin{lemma}\label{lemma:ho(Env, A-)}
For any place $v_n \mid p$, $\ho(K_{n, v}, A^-) = 0$.
\end{lemma}

\begin{proof}
In the isomorphism $A^- \cong \K/\O$, $\Frob_p \in G_{n, v}$ acts as $\alpha$ and $(\alpha - 1) \in \O^\times$, as we saw in the previous lemma. 
Therefore for any $x + \O \in \K/\O$, with $v_\p(x) < 0$, 
\[
\Frob_p \cdot \, (x + \O) = \alpha x + \O \ne x + \O,
\]
because $v_\p(\alpha x - x) =  v_\p\big((\alpha - 1) x \big) = v_\p(x) < 0$. This means that there exists for any nonzero element of $A^-$ an 
automorphism of $G_{n, v}$ that does not fix it and so $\ho(K_{n, v}, A^-) = 0$.
\end{proof}

We may now compare the various Selmer groups of $A$ that we know. We first deal with the comparison of finite and unramified local conditions at 
primes $v$ of bad reduction, i.e.~$v \mid N$.
\begin{lemma}
For any place $v_n$ of $K_n$ such that $v_n \mid N$, we have $c_{v_n}(A) = 1$. 
\end{lemma}

\begin{proof}
For any $p$-adic field $L$ write $\mathcal{A}_L = \ho(I_L, A)/\ho(I_L, A)_\divisible$. By 
\cite[Lemma I.3.2(\emph{iii})]{rubin:euler-systems}
\[
\frac{\honeur(L, A)}{\honef(L, A)} \iso \frac{\mathcal{A}_L}{(\Frob_L - 1)\mathcal{A}_L}.
\]
Since by Assumption \ref{assumption:2}, $c_v(A) = 1$ for any $v \mid N$, then $\honeur(K_v, A) = {\honef(K_v, A)}$ and hence 
${\mathcal{A}_{K_v}}/{(\Frob_{K_v} - 1)\mathcal{A}_{K_v}} = 0$. But since $v \nmid p$, $K_n/K$ is unramified at $v$, hence $I_{n, v} = I_v$ and 
therefore $\mathcal{A}_{K_{n, v}} = \mathcal{A}_{K_v}$. Thus $\honeur(K_{n, v}, A) = {\honef(K_{n, v}, A)}$.
\end{proof}

\begin{proposition}\label{prop:comparison-h1}
$\honetildef(K_n, A) = \honef(K_n , A) = \hone_\str(K_n, A) = \hone_\Gr(K_n, A)$ for any $n \ge 0$.
\end{proposition}

\begin{proof}
Observe that $\honetildef(K_n, A) = \hone_\str(K_n, A)$, by the exact sequence of Proposition \ref{prop:comparison-generalized-greenberg} applied to $E = K_n$
as $\ho(K_{n, v}, A^-) = 0$ for any $v_n \mid p$ by Lemma \ref{lemma:ho(Env, A-)}. Moreover by Lemma \ref{lemma:hone(ho(Inv, A-))} the exact sequence
\[
0 \to \hone_\str(K_n , A) \to \hone_\Gr(K_n, A) \to \bigoplus_{v_n \mid \, p} \hone\big(G_{n, v}/I_{n, v}, \ho(I_{n, v}, A^-)\big)  = 0
\]
shows that $\hone_\str(K_n, A) = \hone_\Gr(K_n, A)$. Finally we have $\hone_\Gr(K_n, A) = \honef(K_n, A)$ by comparing each local condition: 
in fact for any $v_n \nmid p$, $\honef(K_{n,v}, A) = \hone_\ur(K_{n, v}, A)$ as $c_{v_n}(A) = 1$  and for $v_n \mid p$ the proof of 
\cite[Lemma 5.4]{longo-vigni:control-theorems} shows (it is the injectivity of the first map considered there) that 
\[
\honef(K_{n, v}, A) = \ker\big(\hone(K_{n, v}, A) \to \hone(I_{n, v}, A^-)\big) = \hone_\ord(K_{n, v}, A).\qedhere
\]
\end{proof}

We may state a more precise theorem about generalized Selmer groups. We need first some technical lemmas about the vanishig of the $\cohomology^0$ of $A$ over 
the anticyclotomic tower.
\begin{proposition}\label{prop:ho-A[p]}
$\ho(K_n, A[p]) = 0$ for any $n \ge 0$.
\end{proposition}

\begin{proof}
It follows by \cite[Lemma 3.10]{longo-vigni:control-theorems} as $K_n/\Q$ is solvable.
\end{proof}

\begin{proposition}\label{prop:ho}
\begin{enumerate}
\item $\ho(K_n, A[p^m]) = 0$ for any $n \ge 0$, $m > 0$;
\item $\ho(K_n, A) = 0$ for any $n \ge 0$;
\item $\ho(K_\infty, A) = \ho(K_\infty, A[p^m]) = 0$ for any $m > 0$;
\item $\ho(K_n, V) = 0$, $\ho(K_n, T) = 0$ for any $n \ge 0$;
\item $\ho(K_n, V^\ast(1)) = 0$, $\ho(K_n, T^\ast(1)) = 0$ for any $n \ge 0$.
\end{enumerate}
\end{proposition}

\begin{proof}
(1) by induction: Proposition \ref{prop:ho-A[p]} is the base case. Suppose then $\ho(K_n, A[p^{m-1}])=0$ for $m > 1$; if $x \in \ho(K_n, A[p^m])$, then 
$px \in \ho(K_n, A[p^{m-1}]) = 0$ by the induction hypothesis. Therefore $x \in \ho(K_n, A[p])$, and so $x = 0$ again by Proposition \ref{prop:ho-A[p]}. 

(2) follows from (1) taking inductive limit over $m$ as $A \cong \injlim_m A[p^m]$ and the fact that $\ho$ commutes with $\injlim$. 
(3) then follows from (2) and (1) taking inductive limit over $n$. 
(4) follows from (2): Suppose that there is an element $0 \ne x \in V$ fixed by any $\sigma \in {G_{K_n}}$, then its image $\bar{x} \in A$ is still fixed. 
We may moreover assume $x \notin T$, up to multiplication by a suitable power of $p$ (the action of $\sigma$ is linear), i.e.~$\bar{x} \ne 0$. 
This proves the claim for $V$. The claim for $T$ follows because $T \subseteq V$, so $\ho(K_n, T) \subseteq \ho(K_n, V)$.
At last by (4) and the isomorphism $V \cong V^\ast(1)$ follows that $\ho(K_n, T^\ast(1)) \subseteq \ho(K_n, V^\ast(1)) = 0$ and hence (5) is proved.
\end{proof}

\begin{proposition}\label{prop:comparison}
For $V$ the self-dual Galois representation attached to a $\p$-ordinary newform, we have  
\[
\tildecohomology^i_f(K_n, A) =
\begin{cases}
\honef(K_n, A) &\text{for $i = 1$};\\
D\Big(\honetildef \big(K_n, T^\ast(1)\big)\Big) &\text{for $i = 2$};\\
0 &\text{for $i \ne 1, 2$}.
\end{cases}
\]
\end{proposition}

\begin{proof}
We have $\tildecohomology^1_f(K_n, A) = \honef(K_n, A)$ by Proposition \ref{prop:comparison-h1}. By Proposition \ref{th:duality-rep}, there is an isomorphism 
$\tildecohomology^i_f(K_n, A) \cong D\big(\tildecohomology^{3 - i}_f(K_n, T^\ast(1))\big)$,  
for any $i \in \Z$. 
Therefore, as $\tildecohomology^{j}_f(K_n, A) = \tildecohomology^{j}_f(K_n, T^\ast(1)) = 0$ if $j < 0$ by definition, it follows that $\tildecohomology^i_f(K_n, A) = 0$ for $i \ne 0, 1, 2, 3$. 
Furthermore $\tildecohomology^3_f(K_n, A) = 0$, as it is the dual of $\tildecohomology^{0}_f(K_n, T^\ast(1))$, that is a submodule of $\ho(K_n, T^\ast(1))$ 
and the latter vanishes by Proposition \ref{prop:ho}(5). Similarly $\tildecohomology_f^0(K_n, A) = 0$ as it is a submodule of $\ho(K_n, A)$, that vanishes 
by Proposition \ref{prop:ho}(2).
\end{proof}

Taking the direct limit over $n$ we have the following:
\begin{proposition}\label{prop:comparison-iwasawa}
For $V$ the self-dual Galois representation attached to a $\p$-ordinary newform  	
\[
\tildecohomology^i_f(K_\infty, A) =
\begin{cases}
\honef(K_\infty, A), &\text{for $i = 1$};\\
D\Big(\honetildef \big(K_\infty, T^\ast(1)\big)\Big), &\text{for $i = 2$};\\
0 &\text{for $i \ne 1, 2$}.
\end{cases}
\]
\end{proposition}

In particular Theorem \ref{th:exact-control-theorem-representations} gives the following result.
\begin{theorem}[Exact Control Theorem]\label{th:exact-control-th}
The canonical map
\[
\honef(K_n, A) \iso \honef(K_\infty, A)^{\Gal(K_\infty/K_n)}
\]
is an isomorphism.
\end{theorem}

\begin{proof}
Apply Theorem \ref{th:exact-control-theorem-representations} to $K_\infty/K_n$ and combine with Proposition \ref{prop:comparison} and \ref{prop:comparison-iwasawa}.
\end{proof}

\begin{remark}\label{rk:coinvariants-new}
Recall that for any $\Lambda$-module $M$ the $\O$-module of co-invariants is defined to be $M_\Gamma = M/IM$, 
where $I$ is the ideal of $\Lambda$ generated by $\gamma - 1$, for a topological generator $\gamma$ of $\Gamma$. 
In terms of coinvariants the exact control theorem is equivalent to an isomorphism of $\Lambda$-modules 
\[
(\calX_\infty)_\Gamma \iso \calX,
\]
as $D(M^\Gamma) \cong D(M)_\Gamma$ (see for instance \cite[Lemma 5.15]{longo-vigni:control-theorems}).
\end{remark}

\subsection{Vanishing of $\shat(f/K_\infty)$}\label{sec:vanishing-sha-anticyclotomic}

In this section we obtain our main theorem. As a first step we show that we may extend the corank one of the Bloch--Kato Selmer 
group over the anticyclotomic extension $K_\infty/K$.
\newpage
\begin{theorem}\label{th:main-selmer-version}
Let $\calX$ be a free $\O$-module of rank $1$, then $\calX_\infty$ is a free $\Lambda$-module of rank $1$.
\end{theorem}
 
\begin{proof}
By Theorem \ref{th:exact-control-th} and Remark \ref{rk:coinvariants-new}, $(\calX_\infty)_{\Gamma} \cong \calX \cong \O$, i.e.~there is an $x \in \calX_\infty$
such that its image in $(\calX_\infty)_\Gamma$ generates it as $\O$-module. Consider the short exact sequence
\[
0 \to x\Lambda \to \calX_\infty \to \calX_\infty/x\Lambda \to 0,
\]
Taking coinvariants, we see that the sequence
\[
(x\Lambda)_\Gamma \to (\calX_\infty)_\Gamma \to (\calX_\infty/x\Lambda)_\Gamma \to 0
\]
is exact. Therefore $({\calX_\infty}/{x \Lambda})_\Gamma = 0$, since the map $(x\Lambda)_\Gamma \to (\calX_\infty)_\Gamma$ is an isomorphism, 
and hence by Nakayama's lemma $\calX_\infty = x\Lambda$, i.e.~$\calX_\infty$ is a cyclic $\Lambda$-module. 

It is left to show that $\calX_\infty$ is not $\Lambda$-torsion. Suppose by contradiction that this is the case. 
Consider the morphism $\eta  \colon \calX_\infty \to \Lambda \oplus M \oplus M$ of Theorem \ref{th:longo-vigni}. The $\Lambda$-module   
$\calX_\infty$ is cyclic and torsion and so $\Im \, \eta = \alpha\Lambda$, with $\alpha$ a torsion element, i.e.~$\alpha \in M \oplus M$. Therefore 
\[
\coker \eta  = \frac{\Lambda \oplus M \oplus M}{\alpha\Lambda} = \Lambda \oplus \frac{M \oplus M}{\alpha\Lambda}
\]
is infinite, contradicting the fact that $\eta$ is a pseudo-isomorphism.
\end{proof}
 
The previous theorem can be rephrased in terms of the ($\p$-primary) Shafarevich--Tate group of $f$.
In order to do that we need to define $\imajt$, the analogous of the Mordell-Weil group of an elliptic curve, 
over $K_\infty$ and its layers $K_n$. 
As for the definition of $\imajt(K)$, in order to define it we assume that the restriction map
\[
\res_{K[p^{n+1}]/K_n} \colon \hone(K_n, T) \to \hone(K[p^{n+1}], T)^{\Gal(K[p^{n+1}]/K_n)}
\]
is an isomorhism for $n > 0$. 
\begin{definition}\label{def:imajt-anticyclotomic-layers}
We define for any $n > 0$,
\[
\imajt(K_n) := \res_{K[p^{n+1}]/K_n}^{-1}\bigl(\, \imajt(K[p^{n+1}])^{\Gal(K[p^{n+1}]/K_n)} \, \bigr).
\]
Note that $\res_{K_{n+1}/K_n} \colon \hone(K_n, T) \to \hone(K_{n+1}, T)$  restricts to a map $\res \colon \imajt(K_n) \inj \imajt(K_{n+1})$ for any $n\ge 0$, injective since $\ho(K_{n}, T) = 0$ by Proposition \ref{prop:ho}. We define then 
 \[\imajt(K_\infty) = {\displaystyle\injlim_{n, \res}} \imajt(K_n).\]
\end{definition}
In particular we have that $\imajt(K_n) \otimes \K/\O \subseteq \honef(K_n, A)$ for $n \ge 0$ and
\[
\imajt(K_\infty) \otimes \K/\O = \injlim_{n} \bigl(\imajt(K_n) \otimes \K/\O\bigr) \subseteq \injlim_n \honef(K_n, A) = \honef(K_\infty, A)
\]
Consider moreover the module of universal norms $N\imajt(K_\infty) = {\displaystyle\projlim_{n, \cores}} \imajt(K_n) \subseteq \honef(K_\infty, T)$.
\begin{definition}\label{def:shat-K-infty}
The $\p$-primary Tate-Shafarevich group $\shat(f/K_n)$ of $f$ over $K_n$ is defined, for any $n = 0, \dots, \infty$, by the exact sequence
\[
\begin{tikzcd}
0 \ar[r] & \imajt(K_n) \otimes \K/\O \ar[r] & \honef(K_n, A) \ar[r] & \shat(f/K_n) \ar[r] & 0.
\end{tikzcd} 
\]
\end{definition}

\begin{remark}
Note that under our assumptions the restriction is indeed an isomorphism: it follows by Lemma \ref{lemma:res-iso}, by the same 
argument of Remark \ref{rk:res-iso-T}. Therefore $\imajt(K_n)$ and $\shat(f/K_n)$ are well defined. 
In particular, since $\ho(K_n, A[p]) = 0$ by Proposition \ref{prop:ho-A[p]}, 
the same argument of Corollary \ref{cor:selmer-image-aj-free} shows that $\imajt(K_n)$ is a free $\O$-module of finite rank.
Moreover each $\imajt(K_n)$ is endowed with an action of $\Gal(K_n/K)$ and hence $\imajt(K_\infty) \otimes \K/\O$ and $N\imajt(K_\infty)$ have a structure of $\Lambda$-module, compatible with the inclusion respectively into $\honef(K_\infty, A)$ and $\hone(K_\infty, T)$.
\end{remark}

Using this language we may enhance Theorem \ref{th:main-selmer-version}.
\begin{theorem}\label{th:main-sha}
Suppose that $\imajt(K) \otimes \K/\O$  and $N\imajt(K_\infty)$ are nontrivial. If $\calX$ is free of rank $1$ over $\O$, then 
\[
\shat(f/K) = \shat(f/K_\infty) = 0, \quad \honef(K_\infty, A) = \imajt(K_\infty) \otimes \K/\O,
\] 
and the Pontryagin dual $\calX_\infty$ of the latter group is free of rank $1$ over $\Lambda$.  
\end{theorem}

\begin{proof}
If $\calX \cong \O$, then $\honef(K, A) \cong \K/\O$. By Corollary~\ref{cor:selmer-image-aj-free}, $\imajt(K) = \O^{j}$ for some $j > 0$, therefore
\[
0 \ne \imajt(K) \otimes \K/\O \cong (\K/\O)^{j} \subseteq \honef(K, A) \cong \K/\O,
\] 
and hence $j = 1$ and the inclusion must be an equality. Thus $\shat(f/K) = 0$.

Consider now $K_\infty$. By Theorem \ref{th:main-selmer-version}, we already know that $\calX_\infty$ has rank $1$ over $\Lambda$; let us show 
that $\shat(f/K_\infty) = 0$. Let
\[
\calZ_\infty = D\bigl( \shat(f/K_\infty) \bigr), \qquad \calY_\infty = D\bigl(\imajt(K_\infty) \otimes \K/\O\bigr).
\]
By duality we have a short exact sequence
\[
\begin{tikzcd}
0 \ar[r] & \calZ_\infty \ar[r] & \calX_\infty \ar[r] & \calY_\infty \ar[r] & 0 
\end{tikzcd}
\]
and therefore our vanishing claim is equivalent to show that $\calY_\infty = \calX_\infty/\calZ_\infty \cong \Lambda/\calZ_\infty$ is not  $\Lambda$-torsion. In the rest of the proof we show that if $\calY_\infty$ were $\Lambda$-torsion, then $N\imajt(K_\infty)=0$, against our assumptions.

First observe that 
\begin{align*}
\calY_\infty &= \Hom_\O\bigl(\imajt(K_\infty) \otimes \K/\O, \K/\O\bigr) \\
&=\Hom_\O\bigl(\imajt(K_\infty), \Hom_\O(\K/\O, \K/\O)\bigr) \\
&=\Hom_\O\bigl(\imajt(K_\infty), \O\bigr) = \imajt(K_\infty)^\ast
\end{align*}
and that in particular $\calY_\infty[p] = \Hom_\O\bigl(\imajt(K_\infty), \O[p]\bigr) = 0$, since $\O[p] = 0$.

It follows that if $\calY_\infty$ were $\Lambda$-torsion, then it would be a free $\O$-module of finite type. 
Indeed, there would be a nonzero ideal $J$ of $\O\llbracket T \rrbracket$ such that $\calY_\infty \cong \O\llbracket T \rrbracket/J$: 
since $\calY_\infty[p]  = 0$ we can find $f \in J$ such that $f \notin p \O\llbracket T \rrbracket$. By the division lemma  (see \cite[Lemma 5.3.1]{neu:cnf}) $\O\llbracket T \rrbracket/(f)$ would be a free $\O$-module of finite rank and hence also its quotient $\O\llbracket T \rrbracket/J$ would be finitely generated over $\O$. It would also be free following the same argument of the proof of Corollary~\ref{cor:selmer-image-aj-free}, since $\calY_\infty[p] = 0$.

Therefore, since $\calY_\infty = \imajt(K_\infty)^\ast$, also $\imajt(K_\infty)$ would be a free $\O$-module of finite rank; 
in particular there would be an $m$ such that $\imajt(K_\infty) = \res_{K_\infty/K_m}\Bigl( \imajt(K_m) \Bigr)$, then
\begin{align*}
\cores_{K_{n+n'}/K_n}\bigl(\imajt(K_{n+n'})\bigr) &= \cores_{K_{n+n'}/K_n} \circ \res_{K_{n+n'}/K_n} \Bigl( \imajt(K_n) \Bigr)\\
&  = [K_{n+n'}: K_n] \imajt(K_n) = p^{n'} \, \imajt(K_n) 
\end{align*}
for any $n'>0, n \ge m$. Therefore $N\imajt(K_\infty) = 0$, contradicting our hypothesis. Consider indeed $(\epsilon_n)_n \in \imajt(K_\infty)$, by the above formula we would have for any $n \ge m$ and $n'>0$ that $\epsilon_n = \cores_{K_{n+n'}/K_n} (\epsilon_{n+n'})$ is divisible by $p^{n'}$: this forces $\epsilon_n = 0$ since $\imajt(K_n)$ 
is a finite free $\O$-module and therefore it does not have infinitely divisible nonzero elements.
Hence $\epsilon_n = 0$ for any $n\ge0$.
\end{proof}

\newpage
Finally we may gather everything together and state our main theorem.
\begin{theorem}\label{th:main}
Suppose that the basic generalized Heegner cycle $z_{f, K}$ is non-torsion and that 
$z_{f, K} \notin p\hone(K, T)$. Then $\shat(f/K)=0$ and 
\[
\imajt(K) \otimes \K/\O = \honef(K, A) = z_{f, K} \cdot \K/\O,
\] 
moreover $\shat(f/K_\infty) = 0$ and $\honef(K_\infty, A) = \imajt(K_\infty) \otimes \K/\O$, and the Pontryagin dual $\calX_\infty$ of the latter group is free of rank $1$ over $\Lambda$.
\end{theorem}

\begin{proof}
This is the combination of Theorem \ref{th:besser} and Theorem \ref{th:main-sha}. Indeed, since $z_{f, K} \in \imajt(K)$ is non-torsion, 
then it is nontrivial in $\imajt(K) \otimes \K/\O$ and $z_{f, K} \cdot \K/\O$ has corank $1$ over $\O$. Moreover $N\imajt(K_\infty) \ne 0$, since it contains an element $\tilde{\kappa}_1$, constructed by Longo and Vigni in \cite{longo-vigni:generalized} from the classes 
$\alpha_n = \cores_{K[p^{n+1}]/K_n}(z_{f, p^{n+1}}) \in \imajt(K_n)$: $\tilde{\kappa}_1 \ne 0$ by \cite[Theorem 4.12]{longo-vigni:generalized}. 
\end{proof}
\printbibliography[heading=bibintoc]

\end{document}